\documentclass[11pt]{article}

\usepackage{amsmath, amsfonts, amssymb, amsthm,  graphicx, mathtools, enumerate}

\usepackage{bbm}
\usepackage[blocks, affil-it]{authblk}

\usepackage{graphicx}

\usepackage[numbers, square]{natbib}
\usepackage[CJKbookmarks=true,
            bookmarksnumbered=true,
			bookmarksopen=true,
			colorlinks=true,
			citecolor=red,
			linkcolor=blue,
			anchorcolor=red,
			urlcolor=blue]{hyperref}
\usepackage[usenames]{color}

\usepackage[letterpaper, left=1.2truein, right=1.2truein, top = 1.2truein, bottom = 1.2truein]{geometry}
\usepackage[ruled, vlined, lined, commentsnumbered]{algorithm2e}

\usepackage{prettyref,soul}

\newtheorem{lemma}{Lemma}[section]

\newtheorem{thm}{Theorem}[section]
\newtheorem{definition}{Definition}[section]

\newrefformat{eq}{(\ref{#1})}
\newrefformat{chap}{Chapter~\ref{#1}}
\newrefformat{sec}{Section~\ref{#1}}
\newrefformat{algo}{Algorithm~\ref{#1}}
\newrefformat{fig}{Fig.~\ref{#1}}
\newrefformat{tab}{Table~\ref{#1}}
\newrefformat{rmk}{Remark~\ref{#1}}
\newrefformat{clm}{Claim~\ref{#1}}
\newrefformat{def}{Definition~\ref{#1}}
\newrefformat{cor}{Corollary~\ref{#1}}
\newrefformat{lmm}{Lemma~\ref{#1}}
\newrefformat{lemma}{Lemma~\ref{#1}}
\newrefformat{prop}{Proposition~\ref{#1}}
\newrefformat{app}{Appendix~\ref{#1}}
\newrefformat{ex}{Example~\ref{#1}}
\newrefformat{exer}{Exercise~\ref{#1}}
\newrefformat{soln}{Solution~\ref{#1}}
\newrefformat{cond}{Condition~\ref{#1}}
\newcommand{\floor}[1]{{\left\lfloor {#1} \right \rfloor}}

\def\Var{\textsf{Var}} 

\def\text#1{\mbox{\rm #1}}

\DeclarePairedDelimiter{\ceil}{\lceil}{\rceil}

\newcommand{\wh}{\widehat}
\newcommand{\wt}{\widetilde}

\newcommand{\TV}{{\sf TV}}

\title{Density Estimation with Contaminated Data: \\
Minimax Rates and Theory of Adaptation
}
\author{Haoyang Liu and Chao Gao}
\affil{
University of Chicago
}

\begin{document}
\maketitle

\begin{abstract}
This paper studies density estimation under pointwise loss in the setting of contamination model. The goal is to estimate $f(x_0)$ at some $x_0\in\mathbb{R}$ with i.i.d. observations,
$$
X_1,\dots,X_n\sim (1-\epsilon)f+\epsilon g,
$$
where $g$ stands for a contamination distribution. In the context of multiple testing, this can be interpreted as estimating the null density at a point. We carefully study the effect of contamination on estimation through the following model indices: contamination proportion $\epsilon$, smoothness of target density $\beta_0$, smoothness of contamination density $\beta_1$, and level of contamination $m$ such that $g(x_0)\leq m$. It is shown that the minimax rate with respect to the squared error loss is of order
$$
[n^{-\frac{2\beta_0}{2\beta_0+1}}]\vee[\epsilon^2(1\wedge m)^2]\vee[n^{-\frac{2\beta_1}{2\beta_1+1}}\epsilon^{\frac{2}{2\beta_1+1}}],
$$
which characterizes the exact influence of contamination on the difficulty of the problem. We then establish the minimal cost of adaptation to contamination proportion, to smoothness and to both of the numbers. It is shown that some small price needs to be paid for adaptation in any of the three cases. Variations of Lepski's method are considered to achieve optimal adaptation.

The problem is also studied when there is no smoothness assumption on the contamination distribution. This setting that allows for an arbitrary contamination distribution is recognized as Huber's $\epsilon$-contamination model.
The minimax rate is shown to be
$$
[n^{-\frac{2\beta_0}{2\beta_0+1}}]\vee [\epsilon^{\frac{2\beta_0}{\beta_0+1}}].
$$
The adaptation theory is also different from the smooth contamination case. While adaptation to either contamination proportion or smoothness only costs a logarithmic factor, adaptation to both numbers is proved to be impossible.

\smallskip

\end{abstract}

\section{Introduction}

Nonparametric density estimation is a well-studied classical topic \citep{silverman1986density,devroyecombinatorial,tsybakov09}. In this paper, we consider this classical statistical task with a modern twist. Instead of assuming i.i.d. observations from a true density $f$, we assume
\begin{equation}
X_1,...,X_n\sim (1-\epsilon)f+\epsilon g,\label{eq:huber}
\end{equation}
where $g$ is a density not related to $f$, and the goal is to estimate $f(x_0)$ at some $x_0\in\mathbb{R}$. In other words, for each observation, there is an $\epsilon$ probability that the observation is sampled from a distribution not related to the density of interest.

This problem naturally appears in both robust statistics and multiple testing literature. In robust statistics literature, $g$ has the name ``contamination", and the task is interpreted as robustly estimating a density $f$ with contaminated data points \citep{chen2016general}. In multiple testing literature, $f$ and $g$ are respectively called null density and alternative density, and the task is interpreted as estimating null density at a point \citep{efron2004large}. In this paper, we use the name ``contamination" to refer to both $g$ and the observations generated from it.

The nature of the problem heavily depends on the assumptions put on $f$ and $g$. When there is no constraint on the contamination distribution $g$, the data generating process (\ref{eq:huber}) is also recognized as Huber's $\epsilon$-contamination model \citep{huber1964robust,huber1965robust}. Recent work on nonparametric estimation in such a setting includes \cite{chen2016general,gao2017robust}, and the influence of contamination on minimax rates is investigated by \cite{chen2015robust,chen2016general}. On the other hand, in the literature of multiple testing, it is more common to put parametric structural assumptions on the alternative $g$, and optimal rates of estimating the null density $f$ are investigated by \cite{jin2007estimating,cai2010optimal}.

In this paper, we explore this problem with connections to nonparametric density estimation literature in mind. Specifically, the density function $f$ is assumed to have a H\"{o}lder smoothness $\beta_0$. Both cases of structured and arbitrary contamination are considered and fundamental limit of this problem is studied by establishing minimax rate. In the structured contamination case, the contamination distribution $g$ is endowed with a $\beta_1$ H\"{o}lder smoothness, and the contamination level at the point $x_0$ is assumed to satisfy $g(x_0)\leq m$. The minimax rate of estimating $f(x_0)$ with respect to the squared error loss is shown to be of order
\begin{equation}
[n^{-\frac{2\beta_0}{2\beta_0+1}}]\vee[\epsilon^2(1\wedge m)^2]\vee[n^{-\frac{2\beta_1}{2\beta_1+1}}\epsilon^{\frac{2}{2\beta_1+1}}].\label{eq:minimax-stru}
\end{equation}
The minimax rate involves three terms, and the influence of contamination on estimation is precisely characterized. The first term $n^{-\frac{2\beta_0}{2\beta_0+1}}$ corresponds to the classical minimax rate of nonparametric estimation when there is no contamination. The second term $\epsilon^2(1\wedge m)^2$ is determined by contamination on $x_0$. It depends on both the contamination proportion $\epsilon$ and the contamination level $m$. The last term $n^{-\frac{2\beta_1}{2\beta_1+1}}\epsilon^{\frac{2}{2\beta_1+1}}$ is caused by contamination on the neighborhood of $x_0$, which is present even if the contamination level $m$ is zero. In the arbitrary contamination case, or equivalently under Huber's $\epsilon$-contamination model, the minimax rate is of order
\begin{equation}
[n^{-\frac{2\beta_0}{2\beta_0+1}}]\vee [\epsilon^{\frac{2\beta_0}{\beta_0+1}}].\label{eq:minimax-arb-con}
\end{equation}
Compared with (\ref{eq:minimax-stru}), the rate (\ref{eq:minimax-arb-con}) is easier to understand in terms of the influence of the contamination. It is interesting to note that even though $\beta_0$ is the smoothness index of $f$, it still appears on the second term in (\ref{eq:minimax-arb-con}). Thus, when the contamination is arbitrary, its influence on estimation is also determined by the smoothness of the target density.

We also thoroughly investigate the theory of adaptation in both settings of contamination models. Depending on specific settings, various adaptation costs are necessary. For the contamination model with structured contamination, when the contamination proportion is unknown, an optimal adaptive procedure can achieve the rate (\ref{eq:minimax-stru}) with $\epsilon^2(1\wedge m)^2$ replaced by $\epsilon^2$. When the smoothness is unknown, an optimal adaptive procedure can achieve the rate (\ref{eq:minimax-stru}) with $n$ replaced by $n/\log n$. Similarly, for the contamination model with arbitrary contamination, the rate (\ref{eq:minimax-arb-con}) can be achieved up to a logarithmic factor when either $\epsilon$ or $\beta_0$ is unknown. On the other hand, however, when both the contamination proportion and the smoothness are unknown, the adaptation theories are completely different for the two contamination models. For structured contamination, the adaptation cost is just the combination of the cost of unknown contamination proportion and that of unknown smoothness. In contrast, for arbitrary contamination, we show that adaptation is simply impossible when both $\epsilon$ and $\beta_0$ are unknown. In other words, it is impossible to adaptively achieve a rate of the form $n^{-r_1(\beta_0)}\vee \epsilon^{r_2(\beta_0)}$ with any two functions $r_1(\cdot)$ and $r_2(\cdot)$.
 
The theory of adaptation in nonparametric functional estimation without contamination is well studied in the literature.
It is shown by \cite{brown1996constrained,lepski1997optimal,cai2006optimal} that a logarithmic factor must be paid for estimating a point of a density function when smoothness is not known. Adaptation costs of estimating other nonparametric functionals have been investigated in \cite{lepskii1991problem,tribouley2000adaptive,johnstone2001chi,cai2003rates,cai2005adaptive}. Compared with the results in the literature, the presence of contamination brings extra complication to the problem of adaptation. It is remarkable that the adaptation cost depends very sensitively on each specific setting and contamination model. The new phenomena revealed in our paper for adaptation with contamination have not been discovered before.

The rest of the paper is organized as follows. The contamination model with structured contamination is studied in Section \ref{sec:ms} and Section \ref{sec:as}. Results of minimax rates and costs of adaptation are given in Section \ref{sec:ms} and Section \ref{sec:as}, respectively. The corresponding theory of contamination model with arbitrary contamination is investigated in Section \ref{sec:ma}. In Section \ref{sec:discussion}, we discuss extensions of our results to multivariate density estimation and a consistent procedure in the hardest scenario where adaptation is impossible. All proofs are given in Section \ref{sec:proof}.

We close this section by introducing notations that will be used later. For $a,b\in\mathbb{R}$, let $a\vee b=\max(a,b)$ and $a\wedge b=\min(a,b)$. For an integer $m$, $[m]$ denotes the set $\{1,2,...,m\}$.
For a positive real number $x$, $\ceil{x}$ is the smallest integer no smaller than $x$ and $\floor{x}$ is the largest integer no larger than $x$. For two positive sequences $\{a_n\}$ and $\{b_n\}$, we write $a_n\lesssim b_n$ or $a_n=O(b_n)$ if $a_n\leq Cb_n$ for all $n$ with some consntant $C>0$ independent of $n$. The notation $a_n\asymp b_n$ means we have both $a_n\lesssim b_n$ and $b_n\lesssim a_n$.
Given a set $S$, $|S|$ denotes its cardinality, and $\mathbbm{1}_S$ is the associated indicator function. We use $\mathbb{P}$ and $\mathbb{E}$ to denote generic probability and expectation whose distribution is determined from the context. The notation $\mathbb{E}(X:S)$ stands for $\mathbb{E}(X\mathbbm{1}_S)$.
The class of infinitely differentiable functions on $\mathbb R$ is denoted by $\mathcal{C}^{\infty}(\mathbb R)$.
For two probability measures $\mathbb{P}$ and $\mathbb{Q}$, the chi-squared divergence is defined as $\chi^2(\mathbb{P},\mathbb{Q})=\int \frac{d\mathbb{P}^2}{d\mathbb{Q}}-1$, and the total variation distance is defined as $\TV(\mathbb{P},\mathbb{Q})=\sup_B|\mathbb{P}(B)-\mathbb{Q}(B)|$. Throughout the paper, $C$, $c$ and their variants denote generic constants that do not depend on $n$. Their values may change from place to place.


\section{Minimax Rates with Structured Contamination}\label{sec:ms}

\subsection{Results and Implications}\label{sec:minimax-s}

Consider i.i.d. observations $X_1,...,X_n\sim (1-\epsilon)f+\epsilon g$.
The goal is to estimate $f$ at a given point. Without loss of generality, we aim to estimate $f(0)$.
In other words, for every $i\in[n]$, we have $X_i\sim f$ with probability $1-\epsilon$ and $X_i\sim g$ with probability $\epsilon$. Thus, there are approximately $n\epsilon$ observations that are not related to the density function $f$, which are referred to as contamination.

To study the fundamental limit of estimating $f$ with contaminated data, we need to specify appropriate regularity conditions on both $f$ and $g$.
We first define the H\"{o}lder class by
$$\Sigma(\beta,L)=\left\{f:\mathbb{R}\rightarrow\mathbb{R}\Bigg|\left|f^{(\floor{\beta})}(x_1)-f^{(\floor{\beta})}(x_2)\right|\leq L|x_1-x_2|^{\beta-\floor{\beta}}\text{ for any }x_1,x_2\in\mathbb{R}\right\}.$$
Here, $\beta$ stands for the smoothness parameter, and $L$ stands for the radius of the function space.
The H\"{o}lder class of density functions is defined as
$$\mathcal{P}(\beta,L)=\left\{f:\mathbb{R}\rightarrow[0,\infty)\Bigg| f\in\Sigma(\beta,L), \int f=1\right\}.$$
Finally, we define the class of mixtures in the form of $(1-\epsilon)f+\epsilon g$ by
$$\mathcal{M}(\epsilon,\beta_0,\beta_1,L_0,L_1,m)=\left\{(1-\epsilon)f+\epsilon g\Big| f\in \mathcal{P}(\beta_0,L_0), g\in \mathcal{P}(\beta_1,L_1), g(0)\leq m\right\}.$$
This class is indexed by several numbers. Throughout the paper, we refer to $\epsilon$ as contamination proportion and $m$ as contamination level at $0$. The pair $(\beta_0,L_0)$ controls the smoothness of the density function $f$ that we want to estimate, and the pair $(\beta_1,L_1)$ controls the smoothness of the contamination density $g$. Among the six numbers, $\epsilon$ and $m$ are allowed to depend on the sample size $n$, but the numbers $\beta_0,\beta_1,L_0,L_1$ are all assumed to be constants that do not depend on $n$ throughout the paper. It is also assumed that $\epsilon\leq 1/2$.

The minimax risk of estimation is defined as (notice that we suppress the dependence on $n$ for $\mathcal{R}$) 
$$\mathcal{R}(\epsilon,\beta_0,\beta_1,L_0,L_1,m)=\inf_{\wh f(0)}\sup_{p(\epsilon,f,g)\in \mathcal M(\epsilon,\beta_0,\beta_1,L_0,L_1,m)}\mathbb{E}_{X_1,\dots,X_n\sim p}\left(\wh f(0)-f(0)\right)^2,$$
where the notation $p(\epsilon,f,g)$ is used to denote the density $(1-\epsilon)f+\epsilon g$.
Later in the paper, we will shorthand $\mathbb{E}_{X_1,\dots,X_n\sim p}$ by $\mathbb{E}_{p^n}$.
Obviously, the minimax risk becomes smaller if $\epsilon$ gets smaller or $n$ gets larger. Besides the role of $\epsilon$ and $n$, the other model indices are also expected to affect the difficulty of the problem, as listed in the following.
\begin{itemize}
  \item The smoothness of $f$: From classical density estimation theory, we know the smoother $f$ is, the easier it is to estimate $f(0)$.
  \item The level of $g(0)$: Intuitively, the smaller $g(0)$ is, the smaller its influence is on $f(0)$, and thus the easier the problem is.
  \item The smoothness of $g$: Intuitively, the smoother $g$ is, the less the contamination effect can spread, and thus the easier it is to account for the effect of $g$ in the contamination model.
\end{itemize}
Now we present the following theorem of minimax rate, that justifies our intuition above.
\begin{thm}\label{thm:minimax-rate}
Under the setting above, we have
\begin{equation}
\mathcal{R}(\epsilon,\beta_0,\beta_1,L_0,L_1,m)\asymp [n^{-\frac{2\beta_0}{2\beta_0+1}}]\vee[\epsilon^2(1\wedge m)^2]\vee[n^{-\frac{2\beta_1}{2\beta_1+1}}\epsilon^{\frac{2}{2\beta_1+1}}].\label{eq:minimax-rate}
\end{equation}
In other words, $\mathcal{R}(\epsilon,\beta_0,\beta_1,L_0,L_1,m)$ can be upper and lower bounded by the right hand side of (\ref{eq:minimax-rate}) up to a constant that only depends on $\beta_0,\beta_1,L_0,L_1$.
\end{thm}
Theorem \ref{thm:minimax-rate} completely characterizes the difficulty of estimating $f(0)$ with contaminated data. The three terms in the rate (\ref{eq:minimax-rate}) have different but very clear meanings. The first term $n^{-\frac{2\beta_0}{2\beta_0+1}}$ is the classical minimax rate of estimating a smooth function at a given point without contamination. The second term $\epsilon^2(1\wedge m)^2$ is proportional to the squared of the product of contamination level and contamination proportion. The last term $n^{-\frac{2\beta_1}{2\beta_1+1}}\epsilon^{\frac{2}{2\beta_1+1}}$ is perhaps the most interesting. Here the effect of $\epsilon$ is powered by an exponent depending on $\beta_1$, and it stands for the interaction between the contamination proportion and the contamination smoothness. The fact that it does not depend on $m$ implies that we have to pay this price with contaminated data even if $g(0)=0$.

To further understand the implications of Theorem \ref{thm:minimax-rate}, we present the following illustrative special cases of the minimax rate (\ref{eq:minimax-rate}). First, when $\epsilon=0$, we get
$$\mathcal{R}(0,\beta_0,\beta_1,L_0,L_1,m)\asymp n^{-\frac{2\beta_0}{2\beta_0+1}}.$$
This is simply the classical minimax rate of estimating $f(0)$ without contamination.

Next, to understand the role of $m$, we consider two extreme cases of $m=0$ and $m=\infty$. From (\ref{eq:minimax-rate}), we have
$$\mathcal{R}(\epsilon,\beta_0,\beta_1,L_0,L_1,0)\asymp [n^{-\frac{2\beta_0}{2\beta_0+1}}]\vee[n^{-\frac{2\beta_1}{2\beta_1+1}}\epsilon^{\frac{2}{2\beta_1+1}}],$$
and
$$\mathcal{R}(\epsilon,\beta_0,\beta_1,L_0,L_1,\infty)\asymp [n^{-\frac{2\beta_0}{2\beta_0+1}}]\vee\epsilon^2.$$
The case of $m=0$ is particularly interesting. It implies $g(0)=0$, and one may expect that the contamination would have no influence on the minimax rate. This intuition is not true because of the term $n^{-\frac{2\beta_1}{2\beta_1+1}}\epsilon^{\frac{2}{2\beta_1+1}}$. Since nonparametric estimation of $f(0)$ also depends on the values of the density function at a neighborhood of $0$, the contamination from $g$ can still have an effect on the neighborhood of $0$ despite that $g(0)=0$. A smaller value of $\beta_1$ allows a greater perturbation by $g$ on the neighborhood of $0$. When $m=\infty$, the minimax rate has a simple form of $[n^{-\frac{2\beta_0}{2\beta_0+1}}]\vee\epsilon^2$. The influence on the minimax rate from contamination is always $\epsilon^2$, regardless of the smoothness $\beta_1$.

Finally, we consider the cases of $\beta_1=0$ and $\beta_1=\infty$. In fact, the H\"{o}lder class $\Sigma(\beta,L)$ with $\beta_1=\infty$ is not well defined, but the discussion below still holds for a sufficiently large constant $\beta_1$. From (\ref{eq:minimax-rate}), we have
$$\mathcal{R}(\epsilon,\beta_0,0,L_0,L_1,m)\asymp [n^{-\frac{2\beta_0}{2\beta_0+1}}]\vee\epsilon^2,$$
and
$$\mathcal{R}(\epsilon,\beta_0,\infty,L_0,L_1,m)\asymp [n^{-\frac{2\beta_0}{2\beta_0+1}}]\vee[\epsilon^2(1\wedge m)^2].$$
The influence of the contamination takes the forms of $\epsilon^2$ and $\epsilon^2(1\wedge m)^2$ for the two extreme cases. This immediately implies that for any values of $\epsilon,\beta_0,\beta_1,L_0,L_1,m$, we have
$$[n^{-\frac{2\beta_0}{2\beta_0+1}}]\vee[\epsilon^2(1\wedge m)^2]\lesssim \mathcal{R}(\epsilon,\beta_0,\beta_1,L_0,L_1,m)\lesssim [n^{-\frac{2\beta_0}{2\beta_0+1}}]\vee\epsilon^2.$$
In other words, the influence of contamination on the minimax rate is sandwiched between $m^2\epsilon^2$ and $\epsilon^2$.

\subsection{Upper Bounds}

The minimax rate (\ref{eq:minimax-rate}) can be achieved by a simple kernel density estimator that takes the form
\begin{equation}
\wh{f}_h(0)=\frac{1}{n(1-\epsilon)}\sum_{i=1}^n\frac{1}{h}K\left(\frac{X_i}{h}\right).\label{eq:def-KDE}
\end{equation}
This estimator is slightly different from the classical kernel density estimator because it is normalized by $\frac{1}{n(1-\epsilon)}$ instead of $\frac{1}{n}$. The knowledge of the contamination proportion $\epsilon$ is very critical to achieve the minimax rate (\ref{eq:minimax-rate}). Later, we will show in Section \ref{sec:un-epsilon} that the minimax rate (\ref{eq:minimax-rate}) cannot be achieved if $\epsilon$ is not known.

We introduce the following class of kernel functions.
\begin{eqnarray*}
\mathcal{K}_{l}(L) &=& \Bigg\{K:\mathbb{R}\rightarrow\mathbb{R}\Big| \int K=1, \int x^jK(x)dx=0 \text{ for all } j\in[l], \\
&& \quad\quad\|K\|_{\infty}\vee \int K^2 \vee \int|x|^l|K(x)|dx\leq L \Bigg\}.
\end{eqnarray*}
The class $\mathcal{K}_l(L)$ collects all bounded and squared integrable kernel functions of order $l$. The number $L>0$ is assumed to be a constant throughout the paper. We refer to \cite{devroyecombinatorial} for examples of kernel functions in the class $\mathcal{K}_l(L)$.

\begin{thm}\label{thm:upperbound}
For the estimator $\wh{f}(0)=\wh{f}_h(0)$ with some $K\in \mathcal{K}_{\left \lfloor{\beta_0\vee \beta_1}\right \rfloor}(L)$ and $h=n^{-\frac{1}{2\beta_0+1}}\wedge n^{-\frac{1}{2\beta_1+1}}\epsilon^{-\frac{2}{2\beta_1+1}}$, we have
$$\sup_{p(\epsilon,f,g)\in \mathcal M(\epsilon,\beta_0,\beta_1,L_0,L_1,m)}\mathbb{E}_{p^n}\left(\wh f(0)-f(0)\right)^2\lesssim [n^{-\frac{2\beta_0}{2\beta_0+1}}]\vee[\epsilon^2(1\wedge m)^2]\vee[n^{-\frac{2\beta_1}{2\beta_1+1}}\epsilon^{\frac{2}{2\beta_1+1}}].$$
\end{thm}
Theorem \ref{thm:upperbound} reveals an interesting choice of the bandwidth
$h=n^{-\frac{1}{2\beta_0+1}}\wedge n^{-\frac{1}{2\beta_1+1}}\epsilon^{-\frac{2}{2\beta_1+1}}$.
Compared with the optimal bandwidth of order $n^{-\frac{1}{2\beta_0+1}}$ in classical nonparametric function estimation, the $h$ in the structured contamination setting is always smaller. The choice of bandwidth is a consequences of the specific bias-variance tradeoff under the structured contamination model. As an interesting contrast, in the case of arbitrary contamination, the optimal choice of bandwidth is always larger than the usual one, see Section \ref{sec:ma}.

The error bound in Theorem \ref{thm:upperbound} can be found through a classical bias-variance tradeoff argument. We can decompose the difference $\wh{f}(0)-f(0)$ as
\begin{equation}
(\wh{f}(0)-\mathbb{E}\wh{f}(0))+\left(\mathbb{E}\wh{f}(0)-f(0)-\frac{\epsilon}{1-\epsilon}g(0)\right)+\frac{\epsilon}{1-\epsilon}g(0).\label{eq:error-decomp}
\end{equation}
Here, the first term is the stochastic error. The second term gives the approximation error of the kernel convolution. The last term is caused by the contamination at $0$. Direct analysis of the three terms gives the bound
\begin{equation}
\mathbb{E}\left(\wh f(0)-f(0)\right)^2\lesssim \frac{1}{nh}+(h^{2\beta_0}+\epsilon^2h^{2\beta_1})+\epsilon^2(m\wedge 1)^2.\label{eq:RD-s}
\end{equation}
Now with the choice $h=n^{-\frac{1}{2\beta_0+1}}\wedge n^{-\frac{1}{2\beta_1+1}}\epsilon^{-\frac{2}{2\beta_1+1}}$, we obtain the error bound in Theorem \ref{thm:upperbound}. For detailed derivation, see the proof of Theorem \ref{thm:upperbound} in Section \ref{sec:pf-upper-structure}.

\subsection{Lower Bounds}

In this section, we study the lower bound part of the minimax rate (\ref{eq:minimax-rate}). We first state a theorem.
\begin{thm}\label{thm:smoothlowerbound}
We have
$$
\mathcal R(\epsilon,\beta_0,\beta_1,L_0,L_1,m)\gtrsim [n^{-\frac{2\beta_0}{2\beta_0+1}}]\vee[\epsilon^2(1\wedge m)^2]\vee[n^{-\frac{2\beta_1}{2\beta_1+1}}\epsilon^{\frac{2}{2\beta_1+1}}].
$$
\end{thm}

The first term $n^{-\frac{2\beta_0}{2\beta_0+1}}$ is the classical minimax lower bound for nonparametric estimation. Thus, we will only give here a overview of how to derive the second and the third terms.
Two specific functions are used as building blocks for our construction, and their definitions and properties are summarized in the following two lemmas.
\begin{lemma}\label{lem:a}
Let $l(x)=e^{-\frac{1}{1-x^2}}\mathbbm{1}_{\{|x|\leq 1\}}$. Define
$$a(x)=\begin{cases}
c_0l(x+1), & -2\leq x\leq 0, \\
c_0l(x-1), & 0\leq x\leq 2.
\end{cases}$$
The constant $c_0$ is chosen so that $\int a=1$. It satisfies the following properties:
\begin{enumerate}
  \item $a$ is an even density function compactly supported on $[-2,2]$.
  \item $a(0)=0$.
  \item For any constants $\beta,L>0$, there exists a constant $c>0$, such that $ca(cx)\in \mathcal{P}(\beta,L)\cap \mathcal{C}^{\infty}(\mathbb R)$.
  \item For any small constant $c>0$, $a$ is uniformly lower bounded by a positive constant on $[-1,-c]\cup [c,1]$, and it is uniformly upper bounded by a positive constant on $\mathbb{R}$.
\end{enumerate}
\end{lemma}
\begin{lemma}\label{lem:b}
Let $l(x)=e^{-\frac{1}{1-x^2}}\mathbbm{1}_{\{|x|\leq 1\}}$. Define
$$b(x)=\begin{cases}
-l\left(4x+3\right), & -1\leq x\leq -\frac{1}{2}, \\
l(2x), & |x|\leq \frac{1}{2},\\
-l\left(4x-3\right), & \frac{1}{2}\leq x\leq 1.
\end{cases}$$
It satisfies the following properties:
\begin{enumerate}
  \item $b$ is an even function compactly supported on $[-1,1]$.
  \item For any $\beta,L>0$, there exists a constant $c>0$ such that $cb\in \Sigma(\beta,L)\cap \mathcal{C}^{\infty}(\mathbb R)$.
  \item $b$ is uniformly lower bounded by a positive constant on $[-\frac{1}{4},\frac{1}{4}]$, and $|b|$ is uniformly upper bounded by a positive  constant on $\mathbb{R}$.
  \item $\int b=0$.
\end{enumerate}
\end{lemma}

Both the proofs of the second and the third terms in the lower bound involve careful constructions of two pairs of densities $(f,g)$ and $(\wt{f},\wt{g})$. In order to show $\mathcal{R}(\epsilon,\beta_0,\beta_1,L_0,L_1,m)\gtrsim \epsilon^2(1\wedge m)^2$, we consider the following constructions,
\begin{align*}
f(x)&=f_0(x),\\
\wt f(x) &=f_0(x)+c_1\frac{\epsilon}{1-\epsilon}(m\wedge 1)b(x),\\
g(x)&=c_2a(c_2x)+c_1(m\wedge 1)b(x),\\
\wt g(x)&=c_2a(c_2x).
\end{align*}
Here, the constants $c_1,c_2$ are chosen so that the constructed functions $f,\wt{f},g,\wt{g}$ are well-defined densities in the desired parameter spaces. It is easy to check that with the above construction,
$$(1-\epsilon)f+\epsilon g=(1-\epsilon)\wt{f}+\epsilon \wt{g}.$$
This implies that with the presence of contamination, an estimator $\wh{f}(0)$ cannot distinguish between the two data generating processes $(1-\epsilon)f+\epsilon g$ and $(1-\epsilon)\wt{f}+\epsilon \wt{g}$. As a consequence, an error of order $|f(0)-\wt{f}(0)|^2\asymp \epsilon^2(1\wedge m)^2$ cannot be avoided.

The derivation of the lower bound $\mathcal{R}(\epsilon,\beta_0,\beta_1,L_0,L_1,m)\gtrsim n^{-\frac{2\beta_1}{2\beta_1+1}}\epsilon^{\frac{2}{2\beta_1+1}}$ is more intricate. Consider the following four functions,
\begin{align*}
f(x) &=f_0(x),\\
\wt f(x) &=f_0(x)+\frac{\epsilon}{1-\epsilon}{\color{blue}c_{2}\left[h^{\beta_0}l\left(\frac{x}{h}\right)-h^{\beta_0}l\left(\frac{2(x-c_{4})}{h}\right)-h^{\beta_0}l\left(\frac{2(x+c_{4})}{h}\right)\right]},\\
g(x)&=c_{1}a(c_{1}x){+\color{blue}c_{2}\left[h^{\beta_0}l\left(\frac{x}{h}\right)-h^{\beta_0}l\left(\frac{2(x-c_{4})}{h}\right)-h^{\beta_0}l\left(\frac{2(x+c_{4})}{h}\right)\right]}{-\color{red}c_{3}\wt h^{\beta_1}b\left(\frac{x}{\wt h}\right)},\\
\wt g(x)&=c_{1}a(c_{1}x),
\end{align*}
where the definitions of the functions $l,a,b$ are given in Lemma \ref{lem:a} and Lemma \ref{lem:b}. Again, the constants $c_1,c_2,c_3,c_4$ are chosen properly so that the constructed functions are well-defined densities in the desired function classes.

A dominant feature of this constructions is that $g$ is a perturbation of $\wt{g}$ with two levels of perturbation, respectively with bandwidth $h$ and $\wt{h}$, while usual lower bound proof in nonparametric estimation involves perturbing a function at a single bandwidth level. The first level of perturbation $h^{\beta_0}l\left(\frac{x}{h}\right)$ serves to cancel the effect of the corresponding perturbation on $f$, while the second perturbation $-\wt{h}^{\beta_1}b\left(\frac{x}{\wt h}\right)$ serves to ensure the constraint of contamination level. Indeed, if we relate $h$ and $\wt{h}$ through the equation $h^{\beta_0}\asymp \wt{h}^{\beta_1}$, then it is direct that $\wt{g}(0)=g(0)=0$. In other words, the constructed contamination density functions $g$ and $\wt{g}$ both have contamination level $0$. An illustration of this construction with a two-level perturbation is given by Figure \ref{fig:hyl}.
\begin{figure}
\caption{An illustration of the construction of $g$.}\label{fig:hyl}
\centering
\includegraphics[width=7.5cm,height=5cm]{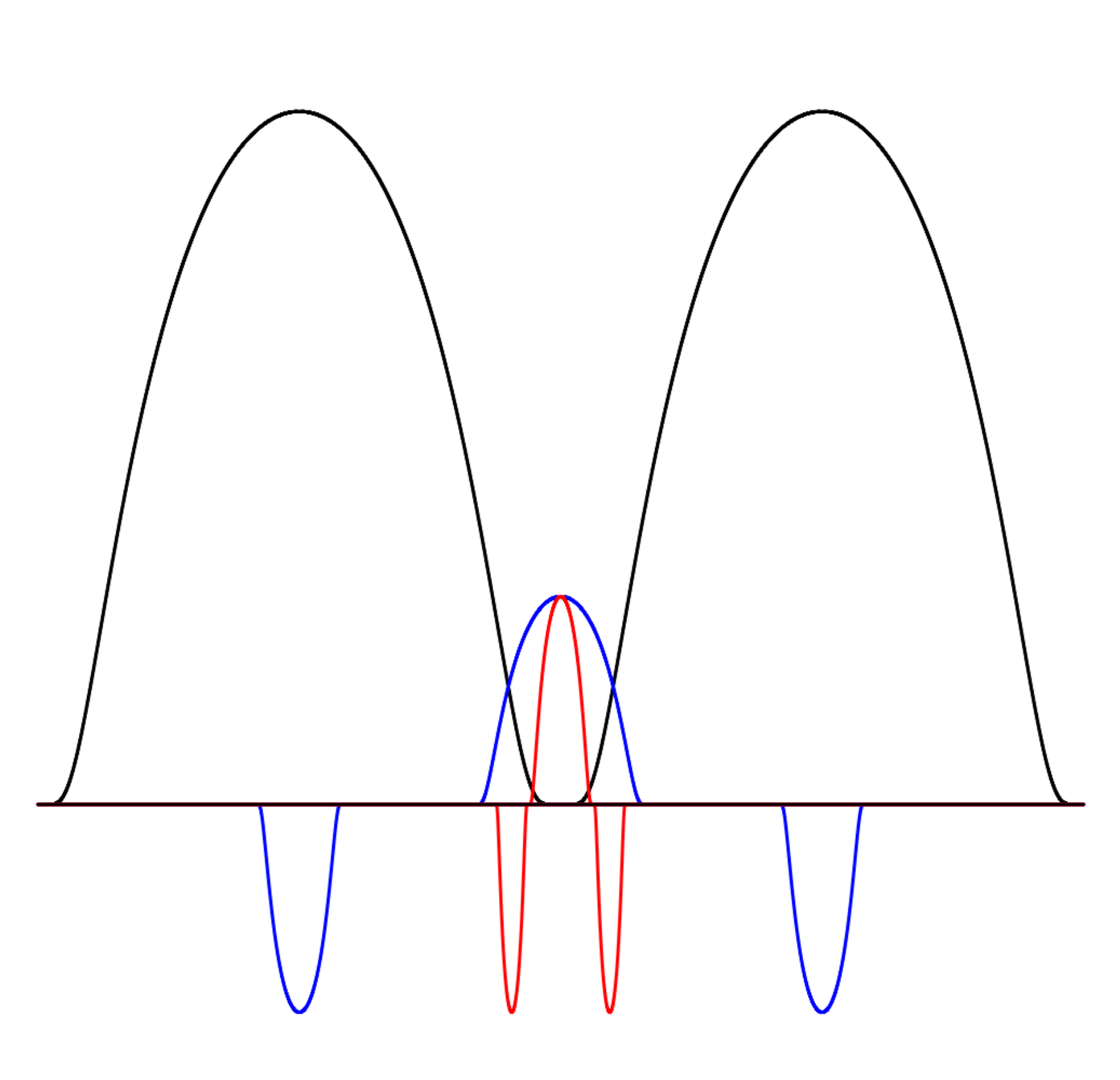}
\end{figure}
The colors of the plot correspond to those in the formulas.

With the above construction, it is not hard to check that
$$p(\epsilon,f,g)-p(\epsilon,\wt{f},\wt{g})=-c_{3}\epsilon\wt h^{\beta_1}b\left(\frac{x}{\wt h}\right).$$
In order that an estimator cannot distinguish between the two densities $p(\epsilon,f,g)=(1-\epsilon)f+\epsilon g$ and $p(\epsilon,\wt{f},\wt{g})=(1-\epsilon)\wt{f}+\epsilon \wt g$, a sufficient condition is $\chi^2\left(p(\epsilon,\wt{f},\wt{g}),p(\epsilon,f,g)\right)\lesssim n^{-1}$ (see Lemma \ref{lem:lowerbound}), which leads to the choice of $\wt{h}$ at the order $\wt{h}\asymp\left(n\epsilon^2\right)^{-\frac{1}{2\beta_1+1}}$. As a consequence, an error of order
$$|f(0)-\wt{f}(0)|^2\asymp \epsilon^2 h^{2\beta_0}\asymp \epsilon^2 \wt{h}^{2{\beta}_1}\asymp\epsilon^{\frac{2}{2\beta_1+1}}n^{-\frac{2\beta_1}{2\beta_1+1}}$$
cannot be avoided. A rigorous proof of Theorem \ref{thm:smoothlowerbound} will be given in Section \ref{sec:pf-lower-structure}.

\section{Adaptation Theory with Structured Contamination}\label{sec:as}

\subsection{Summary of Results}

To achieve the minimax rate in Theorem \ref{thm:minimax-rate}, the kernel density estimator (\ref{eq:def-KDE}) requires the knowledge of contamination proportion $\epsilon$ and smoothness $(\beta_0,\beta_1)$. In this section, we discuss adaptive procedures to estimate $f(0)$ without the knowledge of these parameters. However, adaptation to $\epsilon$ or to $(\beta_0,\beta_1)$ is not free, and one can only achieve slower rates than the minimax rate (\ref{eq:minimax-rate}). The adaptation cost varies for each different scenario. A summary of our results is listed below.
\begin{itemize}
  \item When the contamination proportion is unknown, the best possible rate is
  \begin{align*}
 n^{-\frac{2\beta_0}{2\beta_0+1}}\vee\epsilon^2.
  \end{align*}
  \item When the smoothness parameters are unknown, the best possible rate is
  \begin{align*}
 \left[\left(\frac{n}{\log n}\right)^{-\frac{2\beta_0}{2\beta_0+1}}\right]\vee  \left[\epsilon^2(1\wedge m)^2\right]\vee\left[\left(\frac{n}{\log n}\right)^{-\frac{2\beta_1}{2\beta_1+1}}\epsilon^{\frac{2}{2\beta_1+1}}\right].
  \end{align*}
  \item
  When both the contamination proportion and the smoothness are unknown, the best possible rate becomes
  \begin{align*}
  \left(\frac{n}{\log n}\right)^{-\frac{2\beta_0}{2\beta_0+1}}\vee\epsilon^2.
  \end{align*}
\end{itemize}
Compared with the minimax rate (\ref{eq:minimax-rate}), the ignorance of the contamination proportion implies that $m$ is replaced by $1$ in the rate, while the ignorance of the smoothness implies that $n$ is replaced by $n/\log n$ in the rate.

\subsection{Unknown Contamination Proportion}\label{sec:un-epsilon}

The kernel density estimator (\ref{eq:def-KDE}) depends on $\epsilon$ in two ways: the normalization through $\frac{1}{n(1-\epsilon)}$ and the optimal choice of bandwidth $h$.
Without the knowledge of $\epsilon$, we consider the following estimator
\begin{equation}
\wh f_h(0)=\frac{1}{n}\sum_{i=1}^n \frac{1}{h}K\left(\frac{X_i}{h}\right).\label{eq:def-KDE-n}
\end{equation}
The first difference between (\ref{eq:def-KDE-n}) and (\ref{eq:def-KDE}) is the normalization. When $\epsilon$ is not given, we can only use $\frac{1}{n}$ in (\ref{eq:def-KDE-n}). Moreover, the choice of $h$ in (\ref{eq:def-KDE-n}) cannot depend on $\epsilon$.
\begin{thm}\label{thm:fixedbandwidth}
For the estimator $\wh{f}(0)=\wh{f}_h(0)$ with some $K\in \mathcal{K}_{\left \lfloor{\beta_0\vee \beta_1}\right \rfloor}(L)$ and $h=n^{-\frac{1}{2\beta_0+1}}$, we have
$$\sup_{p(\epsilon,f,g)\in \mathcal M(\epsilon,\beta_0,\beta_1,L_0,L_1,m)}\mathbb{E}_{p^n}\left(\wh f(0)-f(0)\right)^2\lesssim n^{-\frac{2\beta_0}{2\beta_0+1}}\vee\epsilon^2.$$
\end{thm}
With the choice $h=n^{-\frac{1}{2\beta_0+1}}$, $\wh{f}_h$ becomes the classical nonparametric density estimator.
The contamination results in an extra $\epsilon^2$ in the rate compared with the classical nonparametric minimax rate, regardless of the values of $m$ and $\beta_1$. Note that in the current setting, the error $\wh f_h(0)-f(0)$ has the following decomposition,
\begin{equation}
(\wh f_h(0)-\mathbb{E}\wh f_h(0))+(\mathbb{E}\wh f_h(0)-(1-\epsilon)f(0)-\epsilon g(0))+\epsilon(g(0)-f(0)).\label{eq:error-decomp2}
\end{equation}
The difference between (\ref{eq:error-decomp}) and (\ref{eq:error-decomp2}) is resulted from different normalizations in (\ref{eq:def-KDE}) and (\ref{eq:def-KDE-n}).
Some standard calculation gives the bound
\begin{align*}
\mathbb{E}(\wh f_h(0)-f(0))^2\lesssim \frac{1}{nh}\vee h^{2\beta_0}\vee \epsilon^2,
\end{align*}
which implies the optimal choice of bandwidth $h=n^{-\frac{1}{2\beta_0+1}}$, and thus the rate in Theorem \ref{thm:fixedbandwidth}. A detailed proof is given in Section \ref{sec:pf-upper-structure}.

In view of the form of the minimax rate (\ref{eq:minimax-rate}), the rate given by Theorem \ref{thm:fixedbandwidth} can be obtained by replacing the $\epsilon^2(1\wedge m)^2$ in (\ref{eq:minimax-rate}) with $\epsilon^2$. A matching lower bound for adaptivity to $\epsilon$ is given by the following theorem.
\begin{thm}\label{thm:smadapt1}
Consider two models $\mathcal{M}(\epsilon,\beta_0,\beta_1,L_0,L_1,m)$ and $\mathcal{M}(\wt\epsilon,\beta_0,\beta_1,L_0,L_1,m)$ with different contamination proportions. For any estimator $\wh{f}(0)$ that satisfies
\begin{align*}
\sup_{p(\epsilon,f,g)\in\mathcal{M}(\wt\epsilon,\beta_0,\beta_1,L_0,L_1,m)}\mathbb{E}_{p^n}\left(\wh f(0)-f(0)\right)^2\leq C\wt\epsilon^2,
\end{align*}
for some constant $C>0$, there must exist another constant $C'>0$, such that for $\epsilon\geq C'\wt\epsilon$, we have
\begin{align*}
\sup_{p(\epsilon,f,g)\in\mathcal{M}(\epsilon,\beta_0,\beta_1,L_0,L_1,m)}\mathbb{E}_{p^n}\left(\wh f(0)-f(0)\right)^2\gtrsim  \epsilon^2.
\end{align*}
\end{thm}
Theorem \ref{thm:smadapt1} shows that it is impossible to achieve a rate that is faster than $\epsilon^2$ even over only two different contamination proportions. The proof of Theorem \ref{thm:smadapt1} relies on the following construction,
\begin{align*}
f&=f_0,\\
g&=c_1a(c_1x),\\
\wt f&= \frac{1-\epsilon}{1-\wt\epsilon}f_0+\frac{\epsilon-\wt\epsilon}{1-\wt\epsilon}c_1a(c_1x),\\
\wt g&=c_1a(c_1x).
\end{align*}
With an appropriate choice of the constant $c_1>0$, we have $(1-\epsilon)f+\epsilon g\in\mathcal{M}(\epsilon,\beta_0,\beta_1,L_0,L_1,m)$ and $(1-\wt{\epsilon})\wt{f}+\wt{\epsilon}\wt g\in \mathcal{M}(\wt\epsilon,\beta_0,\beta_1,L_0,L_1,m)$. Moreover, it is easy to check that
$$(1-\epsilon)f+\epsilon g=(1-\wt{\epsilon})\wt{f}+\wt{\epsilon}\wt g.$$
In other words, a model with contamination proportion $\epsilon$ can also be written as a mixture that uses a different $\wt{\epsilon}$. Unless the contamination proportion is specified, one cannot tell the difference between $(1-\epsilon)f+\epsilon g$ and $(1-\wt{\epsilon})\wt{f}+\wt{\epsilon}\wt g$. This leads to a lower bound of the error, which is of order $|f(0)-\wt{f}(0)|^2\asymp\epsilon^2$. A rigorous proof of Theorem \ref{thm:smadapt1} that uses a constrained risk inequality in \cite{brown1996constrained} is given in Section \ref{sec:pf-cri}.

\subsection{Unknown Smoothness}\label{sec:un-beta}

In this section, we consider the case that the smoothness numbers are unknown, but the contamination proportion is given.
In view of the kernel density estimator (\ref{eq:def-KDE}) that achieves the minimax rate, we can still use the normalization by $\frac{1}{n(1-\epsilon)}$ because of the knowledge of $\epsilon$, but the bandwidth $h$ needs to be picked in a data-driven way. For a given $h$, define
\begin{align*}
\wh f_h(0)=\frac{1}{n(1-\epsilon)}\sum_{i=1}^n \frac{1}{h}K\left(\frac{X_i}{h}\right).
\end{align*}
With a discrete set $\mathcal{H}$ and some constant $c_1>0$, Lepski's method \citep{lepskii1991problem,lepskii1992asymptotically,lepskii1993asymptotically} selects a data-driven bandwidth through the following procedure,
\begin{equation}
\wh h=\max\left\{h\in\mathcal{H}:|\wh f_h(0)-\wh f_l(0)|\leq c_1\sqrt{\frac{\log n}{nl}}, \forall l\leq h, l\in\mathcal{H}\right\}.\label{eq:h-lep}
\end{equation}
In words, we choose the largest bandwidth below which the variance dominates. If the set that is maximized over is empty, we will use the convention $\wh{h}=\frac{1}{n}$. The estimator $\wh{f}_{\wh{h}}(0)$ that uses a data-driven bandwidth enjoys the following guarantee.
\begin{thm}\label{thm:lepski2}
Consider the adaptive kernel density estimator $\wh{f}(0)=\wh{f}_{\wh{h}}(0)$ with the bandwidth defined by (\ref{eq:h-lep}). In (\ref{eq:h-lep}), we set $\mathcal{H}=\left\{1,\frac{1}{2},\cdots,\frac{1}{2^m}\right\}$ such that $\frac{1}{2^m}\leq \frac{1}{n}<\frac{1}{2^{m-1}}$ and $c_1$ to be a sufficiently large constant. The kernel $K$ is selected from $\mathcal{K}_l(L)$ with a large constant $l\geq \floor{\beta_0\vee\beta_1}$. Then, we have
\begin{eqnarray*}
&& \sup_{p(\epsilon,f,g)\in \mathcal M(\epsilon,\beta_0,\beta_1,L_0,L_1,m)}\mathbb{E}_{p^n}\left(\wh f(0)-f(0)\right)^2 \\
&\lesssim&  \left[\left(\frac{n}{\log n}\right)^{-\frac{2\beta_0}{2\beta_0+1}}\right]\vee  \left[\epsilon^2(1\wedge m)^2\right]\vee\left[\left(\frac{n}{\log n}\right)^{-\frac{2\beta_1}{2\beta_1+1}}\epsilon^{\frac{2}{2\beta_1+1}}\right].
\end{eqnarray*}
\end{thm}
Lepski's method is known to be adaptive over various nonparametric classes, and it can achieve minimax rates up to a logarithmic factor without knowing the smoothness parameter \citep{lepski1997optimal}. Theorem \ref{thm:lepski2} shows that this is also the case with contaminated observations. With an adaptive kernel density estimator normalized by $\frac{1}{n(1-\epsilon)}$, the minimax rate (\ref{eq:minimax-rate}) is achieved up to a logarithmic factor in Theorem \ref{thm:lepski2}.

A comparison between the adaptive rate given by Theorem \ref{thm:lepski2} and the minimax rate (\ref{eq:minimax-rate}) reveals two differences. The first adaptation cost is given by $\left(\frac{n}{\log n}\right)^{-\frac{2\beta_0}{2\beta_0+1}}$, compared with $n^{-\frac{2\beta_0}{2\beta_0+1}}$ in (\ref{eq:minimax-rate}). Previous work in adaptive nonparametric estimation \citep{brown1996constrained,lepski1997optimal,cai2003rates} implies that this cost is unavoidable for adaptation to smoothness.
The second adaptation cost is given by $\left(\frac{n}{\log n}\right)^{-\frac{2\beta_1}{2\beta_1+1}}\epsilon^{\frac{2}{2\beta_1+1}}$, compared with $n^{-\frac{2\beta_1}{2\beta_1+1}}\epsilon^{\frac{2}{2\beta_1+1}}$ in (\ref{eq:minimax-rate}). In the next theorem, we show that this adaptations cost is also unavoidable without the knowledge of the smoothness parameters.

\begin{thm}\label{thm:smadapt2}
Consider two models $\mathcal{M}(\epsilon,\beta_0,\beta_1,L_0,L_1,m)$ and $\mathcal{M}(\epsilon,\wt\beta_0,\wt\beta_1,\wt L_0,\wt L_1,m)$ with different smoothness parameters. Assume that $\beta_0\leq\tilde \beta_0$, $\beta_1<\tilde \beta_1$, ${\beta}_0\geq{\beta}_1$ and $n\epsilon^2\geq(\log n)^2$. For any estimator $\wh{f}(0)$ that satisfies
$$\sup_{p(\epsilon,f,g)\in \mathcal{M}(\epsilon,\wt\beta_0,\wt\beta_1,\wt L_0,\wt L_1,m)}\mathbb{E}_{p^n}\left(\wh f(0)-f(0)\right)^2\leq  C\left(\frac{n}{\log n}\right)^{-\frac{2\wt\beta_1}{2\wt\beta_1+1}}\epsilon^{\frac{2}{2\wt\beta_1+1}},$$
for some constant $C>0$, we must have
$$\sup_{p(\epsilon,f,g)\in \mathcal{M}(\epsilon,\beta_0,\beta_1,L_0,L_1,m)}\mathbb{E}_{p^n}\left(\wh f(0)-f(0)\right)^2\gtrsim \left(\frac{n}{\log n}\right)^{-\frac{2\beta_1}{2\beta_1+1}}\epsilon^{\frac{2}{2\beta_1+1}}.$$
\end{thm}
Similar to the statement of Theorem \ref{thm:smadapt1}, Theorem \ref{thm:smadapt2} shows that it is impossible to achieve a rate that is faster than $\left(\frac{n}{\log n}\right)^{-\frac{2\beta_1}{2\beta_1+1}}\epsilon^{\frac{2}{2\beta_1+1}}$ across two function classes with different smoothness parameters.
We remark that the assumptions ${\beta}_0\geq {\beta}_1$ and $n\epsilon^2\geq(\log n)^2$ in Theorem \ref{thm:smadapt2} are necessary conditions for $\left(\frac{n}{\log n}\right)^{-\frac{2\beta_1}{2\beta_1+1}}\epsilon^{\frac{2}{2\beta_1+1}}$ to dominate $\left(\frac{n}{\log n}\right)^{-\frac{2\beta_0}{2\beta_0+1}}$. Without these two conditions, $\left(\frac{n}{\log n}\right)^{-\frac{2\beta_0}{2\beta_0+1}}$ is the larger term between the two, and the lower bound is already in the literature.

In conclusion, the rate in Theorem \ref{thm:lepski2} achieved by Lepski's method cannot be improved unless smoothness parameters are given.

\subsection{Unknown Contamination Proportion and Unknown Smoothness}

When both the contamination proportion and the smoothness are unknown, we consider Lepski's method with a kernel density estimator normalized by $\frac{1}{n}$. Define
\begin{align*}
\wh f_h(0)=\frac{1}{n}\sum_{i=1}^n \frac{1}{h}K\left(\frac{X_i}{h}\right).
\end{align*}
Then, a data-driven bandwidth $\wh{h}$ is selected according to (\ref{eq:h-lep}). Again, if the set that is maximized over is empty in (\ref{eq:h-lep}), we will use the convention $\wh{h}=\frac{1}{n}$. Note that this is a fully data-driven estimator that is adaptive to both the contamination proportion and the smoothness. It enjoys the following guarantee.

\begin{thm}\label{thm:lepski}
Consider the adaptive kernel density estimator $\wh{f}(0)=\wh{f}_{\wh{h}}(0)$ with the bandwidth defined by (\ref{eq:h-lep}). In (\ref{eq:h-lep}), we set $\mathcal{H}=\left\{1,\frac{1}{2},\cdots,\frac{1}{2^m}\right\}$ such that $\frac{1}{2^m}\leq \frac{1}{n}<\frac{1}{2^{m-1}}$ and $c_1$ to be a sufficiently large constant. The kernel $K$ is selected from $\mathcal{K}_l(L)$ with a large constant $l\geq \floor{\beta_0\vee\beta_1}$. Then, we have
$$\sup_{p(\epsilon,f,g)\in \mathcal M(\epsilon,\beta_0,\beta_1,L_0,L_1,m)}\mathbb{E}_{p^n}\left(\wh f(0)-f(0)\right)^2\lesssim \left(\frac{n}{\log n}\right)^{-\frac{2\beta_0}{2\beta_0+1}}\vee\epsilon^2.$$
\end{thm}

Compared with the minimax rate in Theorem \ref{thm:minimax-rate}, the rate in Theorem \ref{thm:lepski} can be understood as replacing $n$ and $\epsilon^2(1\wedge m)^2$ respectively by $n/\log n$ and $\epsilon^2$ in (\ref{eq:minimax-rate}). In view of the results in both Section \ref{sec:un-epsilon} and Section \ref{sec:un-beta}, this rate $\left(\frac{n}{\log n}\right)^{-\frac{2\beta_0}{2\beta_0+1}}\vee\epsilon^2$ in Theorem \ref{thm:lepski} cannot be improved by any procedure that is adaptive to both contamination proportion and smoothness.


\section{Results for Arbitrary Contamination}\label{sec:ma}

\subsection{Minimax Rates}

In this section, we study the contamination model without any structural assumption on the contamination distribution:
\begin{align*}
X_1,\dots,X_n \sim (1-\epsilon)P_f+\epsilon G
\end{align*}
where $P_f$ is a distribution on $\mathbb{R}$ that has a density function $f$, and $G$ is an arbitrary contamination distribution. This leads to the following model space
$$\mathcal{M}(\epsilon,\beta_0,L_0)=\left\{(1-\epsilon)P_f+\epsilon G\Big| f\in\mathcal{P}(\beta_0,L_0)\text{ and } G\text{ is an arbitrary distribution}\right\}.$$
This is often referred to as Huber's $\epsilon$-contamination model \citep{huber1964robust,huber1965robust}. Nonparametric function estimation under Huber's $\epsilon$-contamination model has recently been studied by \cite{chen2016general,gao2017robust} for global loss functions. In this paper, our focus is on the local estimation of $f(0)$. The corresponding minimax risk is defined by
$$\mathcal{R}(\epsilon,\beta_0,L_0)=\inf_{\wh f(0)}\sup_{p(\epsilon,f,g)\in {\mathcal M}(\epsilon,\beta_0,L_0)}\mathbb{E}_{p^n}\left(\wh f(0)-f(0)\right)^2.$$
In contrast to the minimax rate studied in Section \ref{sec:minimax-s}, we only have one parameter $\epsilon$ that indexes the influence of the contamination for $\mathcal{R}(\epsilon,\beta_0,L_0)$.

\begin{thm}\label{thm:minimax-arb}
Under the setting above, we have
\begin{equation}
\mathcal{R}(\epsilon,\beta_0,L_0)\asymp [n^{-\frac{2\beta_0}{2\beta_0+1}}]\vee[\epsilon^{\frac{2\beta_0}{\beta_0+1}}].\label{eq:easy-minimax}
\end{equation}
\end{thm}
The minimax rate given by Theorem \ref{thm:minimax-arb} only involves two terms. The first term $n^{-\frac{2\beta_0}{2\beta_0+1}}$ is the classical minimax rate for nonparametric estimation. The second term $\epsilon^{\frac{2\beta_0}{\beta_0+1}}$ characterizes the influence of contamination. It is worth noticing that the smoothness index of $f$ appears both in $n^{-\frac{2\beta_0}{2\beta_0+1}}$ and $\epsilon^{\frac{2\beta_0}{\beta_0+1}}$. A larger value of $\beta_0$ implies a less influence of the contamination. This is in contrast to the rate of $\mathcal{R}(\epsilon,\beta_0,\beta_1,L_0,L_1,m)$ in Theorem \ref{thm:minimax-rate}.

The phase transition boundary of $\mathcal{R}(\epsilon,\beta_0,L_0)$ occurs at $\epsilon=n^{-\frac{\beta_0+1}{2\beta_0+1}}$. Below this level, we have $\mathcal{R}(\epsilon,\beta_0,L_0)\asymp n^{-\frac{2\beta_0}{2\beta_0+1}}$, and the contamination has no influence on the classical minimax rate. When $\epsilon$ is above $n^{-\frac{\beta_0+1}{2\beta_0+1}}$, the rate becomes $\epsilon^{\frac{2\beta_0}{\beta_0+1}}$, dominated by the contamination of data. Since we have about $n\epsilon$ contaminated observations in expectation, an optimal procedure can achieve the classical minimax rate $n^{-\frac{2\beta_0}{2\beta_0+1}}$ with at most $n\epsilon\leq n^{\frac{\beta_0}{2\beta_0+1}}$ contaminated data points. Note that the number $n^{\frac{\beta_0}{2\beta_0+1}}$ is an increasing function of $\beta_0$.

For the upper bound of the minimax rate, we again consider the kernel density estimator
$\wh{f}_h(0)=\frac{1}{n}\sum_{i=1}^n\frac{1}{h}K\left(\frac{X_i}{h}\right)$.
The error $\wh{f}_h(0)-f(0)$ can be decomposed as $(\wh{f}_h(0)-\mathbb{E}\wh{f}_h(0))+(\mathbb{E}_h\wh{f}(0)-f(0))$. Then, a direct analysis shows that the risk can be bounded by three terms,
\begin{equation}
\mathbb{E}\left(\wh{f}_h(0)-f(0)\right)^2\lesssim \frac{1}{nh}\vee h^{2\beta_0}\vee \frac{\epsilon^2}{h^2},\label{eq:RD-for-arb}
\end{equation}
which leads to the optimal choice of bandwidth $h=n^{-\frac{1}{2\beta_0+1}}\vee \epsilon^{\frac{1}{\beta_0+1}}$. It is interesting to note that this choice of bandwidth is always larger than or equal to $n^{-\frac{1}{2\beta_0+1}}$. Recall that when the contamination is smooth, the optimal bandwidth in Theorem \ref{thm:upperbound} is smaller than $n^{-\frac{1}{2\beta_0+1}}$. Thus, when there is contamination in the data, one may need to use a larger or smaller bandwidth compared with $n^{-\frac{1}{2\beta_0+1}}$ depending on the assumption of contamination.

The lower bound part of Theorem \ref{thm:minimax-arb} can be viewed as an application of Theorem 5.1 in \cite{chen2015robust}. A general lower bound for Huber's $\epsilon$-contamination model in \cite{chen2015robust} reveals a critical quantity called modulus of continuity, defined as
$$\omega(\epsilon)=\sup\left\{|f(0)-\wt{f}(0)|^2: \TV(P_f,P_{\wt{f}})\leq \epsilon/(1-\epsilon), f,\wt{f}\in\mathcal{P}(\beta_0,L_0)\right\}.$$
The definition of modulus of continuity goes back to \cite{donoho1994statistical,donoho1991geometrizing}, and its relation to Huber's $\epsilon$-contamination model is characterized in \cite{chen2015robust}. In the current setting, it can be shown that $\omega(\epsilon)\asymp \epsilon^{\frac{2\beta_0}{\beta_0+1}}$, which leads to the lower bound part of Theorem \ref{thm:minimax-arb}. In Section \ref{sec:pf-miss}, we will give an alternative self-contained proof of the lower bound.

\subsection{Adaptation to Either Contamination Proportion or Smoothness}\label{sec:ad-hb}

The key to adaptation to either contamination proportion or smoothness is the risk decomposition (\ref{eq:RD-for-arb}) of the kernel density estimator $\wh{f}_h(0)=\frac{1}{n}\sum_{i=1}^n\frac{1}{h}K\left(\frac{X_i}{h}\right)$. We write  (\ref{eq:RD-for-arb}) as the sum of two terms. That is,
\begin{equation}
\frac{1}{nh}\vee h^{2\beta_0}\vee \frac{\epsilon^2}{h^2}\asymp \left(\frac{\epsilon^2}{h^2}+\frac{1}{nh}\right) + h^{2\beta_0}.\label{eq:increasing-decreasing}
\end{equation}
The first term $\frac{\epsilon^2}{h^2}+\frac{1}{nh}$ is a decreasing function of $h$ with a possibly unknown $\epsilon$, while the second term $h^{2\beta_0}$ is an increasing function of $h$ with a possibly unknown $\beta_0$. If we know $\epsilon$ but do not know $\beta_0$, then we can use Lespki's method with $\frac{\epsilon^2}{h^2}+\frac{1}{nh}$ as a reference curve. On the other hand, if we know $\beta_0$ but do not know $\epsilon$, we can then use a reverse version of Lepski's method with $h^{2\beta_0}$ as a reference curve. Specifically, when $\epsilon$ is known but $\beta_0$ is unknown, we use
\begin{equation}
\wh h=\max\left\{h\in\mathcal{H}:|\wh f_h(0)-\wh f_l(0)|\leq c_1\left(\sqrt{\frac{\log n}{nl}}+\frac{\epsilon}{l}\right), \forall l\leq h, l\in\mathcal{H}\right\}.\label{eq:h-lep-epsilon}
\end{equation}
If the set that is maximized over is empty, we take $\wh{h}=\frac{1}{n}$.
When $\beta_0$ is known but $\epsilon$ is unknown, we use
\begin{equation}
\wh h=\min\Bigg\{h\in \mathcal{H}:|\wh f_h(0)-\wh f_l(0)|\leq c_1l^{\beta_0}, \forall l\geq h, l\in\mathcal{H}\Bigg\}.\label{eq:h-lep-beta}
\end{equation}
If the set that is minimized over is empty, we take $\wh{h}=1$.

Before stating the guarantee for $\wh f_{\wh h}(0)$, we want to emphasize that whether the contamination proportion $\epsilon$ is known or not is more than a matter of normalization. As a comparison, recall the risk decomposition for a kernel density estimator with structured contamination in (\ref{eq:RD-s}). There, both $h^{2\beta_0}$ and $\epsilon^2h^{2\beta_1}$ are increasing functions of $h$. This implies that simultaneous adaptation to both $\epsilon$ and $h$ is possible through Lepski's method, and whether $\epsilon$ is given or not only affects the normalization of the kernel density estimator, which is not the case for arbitrary contamination because of (\ref{eq:increasing-decreasing}).

\begin{thm}\label{thm:lepski3}
Consider the adaptive kernel density estimator $\wh{f}(0)=\wh{f}_{\wh{h}}(0)$ with the bandwidth $\wh{h}$ given by (\ref{eq:h-lep-epsilon}) or (\ref{eq:h-lep-beta}). In either case, we set $\mathcal{H}=\left\{1,\frac{1}{2},\cdots,\frac{1}{2^m}\right\}$ such that $\frac{1}{2^m}\leq \frac{1}{n}<\frac{1}{2^{m-1}}$ and $c_1$ to be a sufficiently large constant. The kernel $K$ is selected from $\mathcal{K}_l(L)$ with a large constant $l\geq \floor{\beta_0}$. Then, we have
$$\sup_{p(\epsilon,f,g)\in \mathcal M(\epsilon,\beta_0,L_0)}\mathbb{E}_{p^n}\left(\wh f(0)-f(0)\right)^2\lesssim \left(\frac{\log n}{n}\right)^{\frac{2\beta_0}{2\beta_0+1}}\vee \epsilon^{\frac{2\beta_0}{\beta_0+1}}.$$
\end{thm}

With one of $\epsilon$ and $\beta_0$ given, Theorem \ref{thm:lepski3} guarantees adaptive estimation with the rate $\left(\frac{\log n}{n}\right)^{\frac{2\beta_0}{2\beta_0+1}}\vee \epsilon^{\frac{2\beta_0}{\beta_0+1}}$. Compared with the minimax rate in Theorem \ref{thm:minimax-arb}, we have an extra logarithmic factor due to the ignorance of either $\epsilon$ or $\beta_0$. This logarithmic factor cannot be removed by any adaptive procedure in view of the results of \cite{brown1996constrained,lepski1997optimal,cai2003rates}.

\subsection{Adaptation to Both Contamination Proportion and Smoothness?}

When both contamination proportion and smoothness are unknown, the adaptation theory with arbitrary contamination is completely different from the case with structured contamination. Since there is no constraint on the contamination distribution, a model with $(\epsilon,\beta_0)$ can also be written as a different model with $(\wt{\epsilon},\wt{\beta}_0)$. As a consequence, we can prove the following lower bound.
\begin{lemma}\label{thm:unidentifiable}
For any constants $c_1,c_2>0$, there exists a constant $c_0$, such that for any $\beta_0,\wt{\beta}_0\leq c_1$, and any $L_0, \wt{L_0}\geq c_2$, and any estimator $\wh{f}(0)$, one of the following lower bounds must be true,
\begin{eqnarray*}
\sup_{p(\epsilon,f,g)\in\mathcal{M}(\epsilon,\beta_0,L_0)}\mathbb{E}_{p^n}\left(\wh{f}(0)-f(0)\right)^2 \geq c_0\epsilon^{\frac{2\wt{\beta}_0}{\wt{\beta}_0+1}}, \\
\sup_{p(0,f,g)\in\mathcal{M}(0,\wt{\beta}_0,\wt{L}_0)}\mathbb{E}_{p^n}\left(\wh{f}(0)-f(0)\right)^2 \geq c_0\epsilon^{\frac{2\wt{\beta}_0}{\wt{\beta}_0+1}}.
\end{eqnarray*}
\end{lemma}
Lemma \ref{thm:unidentifiable} says that in order for any estimator to adapt to two classes with different contamination proportions and smoothness indices, say $\mathcal{M}(\epsilon,\beta_0,L_0)$ and $\mathcal{M}(0,\wt{\beta}_0,\wt{L}_0)$, it is impossible to achieve a rate that is better than $\epsilon^{\frac{2\wt{\beta}_0}{\wt{\beta}_0+1}}$ across both classes. The lower bound $\epsilon^{\frac{2\wt{\beta}_0}{\wt{\beta}_0+1}}$ is a function of both $\epsilon$, the contamination proportion of the first class $\mathcal{M}(\epsilon,\beta_0,L_0)$, and $\wt{\beta}_0$, the smoothness index of the second class $\mathcal{M}(0,\wt{\beta}_0,\wt{L}_0)$. As we will show in the following, this specific form has a profound implication, in that an adaptive estimation rate that is a function of an individual class is impossible!

As a first step, the following definition formulates what adaptivity means in our specific setting.
\begin{definition}
An estimator $\wh f(0)$ is called $(c_1,c_2,c_3,r_1(\cdot),r_2(\cdot))$ rate adaptive if the following holds: for any $n\geq 1$, any $\epsilon\leq 1/2$, any $\beta_0\leq c_1$ and any $L_0\leq c_2$, we have
\begin{equation}
\sup_{p(\epsilon,f,g)\in\mathcal{M}(\epsilon,\beta_0,L_0)}\mathbb{E}_{p^n}\left(\wh{f}(0)-f(0)\right)^2\leq c_3n^{-r_1(\beta_0)}\vee\epsilon^{r_2(\beta_0)}.\label{eq:adaptive-def}
\end{equation}
\end{definition}
 As concrete examples, when the contamination distribution is restricted to those with density functions that are H\"{o}lder smooth, it is shown in Theorem \ref{thm:lepski} that adaptive estimation is possible with some $r_1(\beta_0)<\frac{2\beta_0}{2\beta_0+1}$ and $r_2(\beta_0)=2$. When the contamination distribution is arbitrary, Theorem \ref{thm:lepski3} shows that adaptive estimation is possible over $(\epsilon,\beta_0)$ if either $\epsilon$ or $\beta_0$ is fixed (known) with some $r_1(\beta_0)<\frac{2\beta_0}{2\beta_0+1}$ and $r_2(\beta_0)=\frac{2\beta_0}{\beta_0+1}$.
 In contrast, the following theorem shows that such a goal is impossible for any $r_1(\cdot)$ and $r_2(\cdot)$ when both $\epsilon$ and $\beta_0$ are unknown.

\begin{thm}\label{thm:impossible}
For any constants $c_1,c_2,c_3>0$ and any positive functions $r_1(\cdot)$ and $r_2(\cdot)$, there is no estimator $\wh f(0)$ that is $(c_1,c_2,c_3,r_1(\cdot),r_2(\cdot))$ rate adaptive.
\end{thm}

The impossibility result of Theorem \ref{thm:impossible} is a consequence of Lemma \ref{thm:unidentifiable}. The lower bound $\epsilon^{\frac{2\wt{\beta}_0}{\wt{\beta}_0+1}}$ in Lemma \ref{thm:unidentifiable} involves an $\epsilon$ and a $\wt{\beta}$ from two different classes. This leads to a contradiction given the definition of adaptivity in (\ref{eq:adaptive-def}). A rigorous proof of this argument will given in Section \ref{sec:pf-imp}.

In conclusion, when the contamination is arbitrary, the theory of adaptation to both contamination proportion and smoothness is qualitatively different from adaptation to only one of them. In comparison, when the contamination is structured, that difference is just quantitative according to the results in Section \ref{sec:as}. Therefore, in order to achieve sensible error rates adaptively in a robust density estimation context, we need to either assume a given contamination proportion, a given smoothness index, or a structured contamination distribution.


\section{Discussion}\label{sec:discussion}

\subsection{Extensions to Multivariate Settings}

The results in the paper can all be extended to robust multivariate density estimation. We define a $d$-dimensional isotropic H\"{o}lder class as follows,
$$\Sigma_d(\beta,L)=\left\{f:\mathbb{R}^d\rightarrow\mathbb{R}\Bigg|\max_{l\in I(\beta)}\left|\nabla_lf(x_1)-\nabla_lf(x_2)\right|\leq L\|x_1-x_2\|^{\beta-\floor{\beta}}\text{ for any }x_1,x_2\in\mathbb{R}^d\right\},$$
where we use $I(\beta)$ to denote the set of multi-indices $\{l=(l_1,...,l_d)\big| l_1+\cdots+l_d=\floor{\beta}\}$. The class of density functions is defined as
$$\mathcal{P}_d(\beta,L)=\left\{f:\mathbb{R}^d\rightarrow[0,\infty)\Bigg| f\in\Sigma_d(\beta,L), \int f=1\right\}.$$
Note that the dimension $d$ is assumed to be a constant. Then, the two contamination models considered in the paper are extended as
$$\mathcal{M}_d(\epsilon,\beta_0,\beta_1,L_0,L_1,m)=\left\{(1-\epsilon)f+\epsilon g\Big| f\in \mathcal{P}_d(\beta_0,L_0), g\in \mathcal{P}_d(\beta_1,L_1), g(0)\leq m\right\},$$
and
$$\mathcal{M}_d(\epsilon,\beta_0,L_0)=\left\{(1-\epsilon)P_f+\epsilon G\Big| f\in\mathcal{P}_d(\beta_0,L_0)\text{ and } G\text{ is an arbitrary distribution}\right\}.$$
Similarly, we can define the corresponding minimax rates $\mathcal R_d(\epsilon,\beta_0,\beta_1,L_0,L_1,m)$ and $\mathcal R_d(\epsilon,\beta_0,L_0)$.
\begin{thm}\label{thm:multi}
For the two contamination models on $\mathbb{R}^d$, we have
$$\mathcal R_d(\epsilon,\beta_0,\beta_1,L_0,L_1,m)\asymp [n^{-\frac{2\beta_0}{2\beta_0+d}}]\vee[\epsilon^2(1\wedge m)^2]\vee[n^{-\frac{2\beta_1}{2\beta_1+d}}\epsilon^{\frac{2d}{2\beta_1+d}}],$$
and
$$\mathcal R_d(\epsilon,\beta_0,L_0)\asymp [n^{-\frac{2\beta_0}{2\beta_0+d}}]\vee [\epsilon^{\frac{2\beta_0}{\beta_0+d}}].$$
\end{thm}
The extra factor of dimension $d$ makes the interpretation of results even more interesting. For example, the phase transition boundary of $\mathcal{R}_d(\epsilon,\beta_0,L_0)$ now occurs at $\epsilon=n^{-\frac{\beta_0+d}{2\beta_0+d}}$. This implies that the influence of contamination becomes more severe as the dimension grows. In contrast, the minimax rate of $\mathcal R_d(\epsilon,\beta_0,\beta_1,L_0,L_1,m)$ leads to a completely different interpretation. For example, when $m\geq 1$, we have
$$\mathcal R_d(\epsilon,\beta_0,\beta_1,L_0,L_1,m)\asymp n^{-\frac{2\beta_0}{2\beta_0+d}}\vee\epsilon^2.$$
The second term $\epsilon^2$ does not change with the dimension $d$, and the phase transition boundary between $n^{-\frac{2\beta_0}{2\beta_0+d}}$ and $\epsilon^2$ is at $\epsilon=n^{-\frac{\beta_0}{2\beta_0+d}}$, which increases with respect to $d$. This suggests that the influence of contamination becomes less severe as $d$ grows. In short, the contamination influence on density estimation can be drastically different in a multivariate setting, depending on whether the contamination distribution is structured or arbitrary.

\subsection{Consistency in the Hardest Scenario}

When there is no constraint on the contamination distribution, adaptation is impossible over both contamination proportion and smoothness in the sense of (\ref{eq:adaptive-def}). One may wonder whether there is still anything to do in such a scenario with almost nothing is assumed.
In this section, we show that consistency is still possible under this hardest scenario.

Before introducing the procedure, we remark that achieving consistency without knowing $\epsilon$ and $\beta_0$ is a non-trivial problem due to the risk decomposition (\ref{eq:RD-for-arb}) for a kernel density estimator. According to (\ref{eq:RD-for-arb}), a choice of bandwidth that leads to consistency must satisfy $nh\rightarrow\infty$, $h\rightarrow 0$ and $h/\epsilon\rightarrow\infty$. Note that the first and the second requirements can be satisfied easily with a choice of $h$ that does not depend on any model parameter. For example, one can choose $h=n^{-1/2}$. However, the third requirement $h/\epsilon\rightarrow\infty$ is problematic without the knowledge of $\epsilon$. For any choice of $h\rightarrow 0$, there is an adversarial $\epsilon$ to make $h/\epsilon\rightarrow\infty$ fail.

Despite the above difficulty, we show that a data-driven bandwidth leads to consistency if we know that the smoothness $\beta_0$ has a lower bound $\wt{\beta}_0$. We consider a kernel density estimator $\wh f_h(0)=\frac{1}{n}\sum_{i=1}^n\frac{1}{h}K\left(\frac{X_i}{h}\right)$. Then, we choose $h$ by the reverse version of Lepskis' method that is similar to (\ref{eq:h-lep-beta}). We define $\wh{h}$ by
\begin{equation}
\wh h=\min\Bigg\{h\in \mathcal{H}:|\wh f_h(0)-\wh f_l(0)|\leq c_1l^{\wt\beta_0}, \forall l\geq h, l\in\mathcal{H}\Bigg\}.\label{eq:h-lep-beta-wt}
\end{equation}
Again, we use the convention that if the set that is minimized over is empty, we take $\wh{h}=1$.
\begin{thm}\label{thm:arbitrarylepski}
Consider the kernel density estimator $\wh{f}(0)=\wh{f}_{\wh{h}}(0)$ with the bandwidth $\wh{h}$ given by (\ref{eq:h-lep-beta-wt}). We set $\mathcal{H}=\left\{1,\frac{1}{2},\cdots,\frac{1}{2^m}\right\}$ such that $\frac{1}{2^m}\leq \frac{1}{n}<\frac{1}{2^{m-1}}$ and $c_1$ to be a sufficiently large constant. The kernel $K$ is selected from $\mathcal{K}_l(L)$ with a large constant $l\geq \floor{\beta_0}$. Then, as $n\rightarrow\infty$ and $\epsilon\rightarrow 0$. we have
$$\sup_{p(\epsilon,f,g)\in \mathcal M(\epsilon,\beta_0,L_0)}\mathbb{E}_{p^n}\left(\wh f(0)-f(0)\right)^2\rightarrow 0,$$
if $\beta_0\geq \wt{\beta}_0$.
\end{thm}
Note that the requirements $n\rightarrow\infty$ and $\epsilon\rightarrow 0$ are necessary conditions of consistency given the minimax rate (\ref{eq:easy-minimax}). The procedure does not require knowledge of $\epsilon$ or $\beta_0$, and thus consistency can be achieved without knowing $\epsilon$ and $\beta_0$ even if adaptation is impossible. The procedure (\ref{eq:h-lep-beta-wt}) uses a conservative $\wt{\beta}_0$ in the reverse version of Lepski's method, and can be viewed as an extension of (\ref{eq:h-lep-beta}) that uses the true smoothness index $\beta_0$.


\section{Proofs}\label{sec:proof}

\subsection{Proofs of Theorem \ref{thm:upperbound}  and Theorem \ref{thm:fixedbandwidth}}\label{sec:pf-upper-structure}

\begin{proof}[Proof of Theorem \ref{thm:upperbound}]
Decompose the error as
\begin{align*}
\wh{f}(0)-f(0)=(\wh{f}(0)-\mathbb{E}\wh{f}(0))+\left(\mathbb{E}\wh{f}(0)-f(0)-\frac{\epsilon}{1-\epsilon}g(0)\right)+\frac{\epsilon}{1-\epsilon}g(0),
\end{align*}
where the first term is the stochastic error, the second term stands for bias, and the third term is the misspecification error caused by contamination.

For the variance term, we have
\begin{align*}
\mathbb{E}(\wh{f}(0)-\mathbb{E}\wh{f}(0))^2=\Var\left(\frac{\sum_{i=1}^n \frac{1}{h}K\left(\frac{X_i}{h}\right)}{n(1-\epsilon)}\right)=\frac{\Var(\frac{1}{h}K(\frac{X}{h}))}{n(1-\epsilon)^2},
\end{align*}
where
\begin{align*}
\Var\left(\frac{1}{h}K\left(\frac{X}{h}\right)\right)\leq \int\frac{1}{h^2}K^2\left(\frac{x}{h}\right)((1-\epsilon)f(x)+\epsilon g(x))dx\lesssim \frac{1}{h}\int \frac{1}{h}K^2\left(\frac{x}{h}\right)dx\lesssim \frac{1}{h}.
\end{align*}
This gives the variance bound
\begin{equation}\label{eq:term1}
\mathbb{E}(\wh{f}(0)-\mathbb{E}\wh{f}(0))^2\lesssim \frac{1}{nh}.
\end{equation}

For the bias term we have
\begin{align*}
\mathbb{E} \wh{f}(0)=\int \frac{1}{h}K\left(\frac{x}{h}\right)f(x)dx+\frac{\epsilon}{1-\epsilon}\int  \frac{1}{h}K\left(\frac{x}{h}\right)g(x)dx.
\end{align*}
Since $f\in\mathcal{P}(\beta_0,L_0)$ and $g\in\mathcal{P}(\beta_1,L_1)$, we have $|\int \frac{1}{h}K\left(\frac{x}{h}\right)(f(x)-f(0))dx|\lesssim h^{\beta_0}$ and $|\int \frac{1}{h}K\left(\frac{x}{h}\right)(g(x)-g(0))dx|\lesssim h^{\beta_1}$. See \cite[Chapter 1.2]{tsybakov09} for an explicit bias calculation. Adding up the two bias bounds, we get
\begin{equation}\label{eq:term2}
\left|\mathbb{E}\wh{f}(0)-f(0)-\frac{\epsilon}{1-\epsilon}g(0)\right|\lesssim h^{\beta_0}+\epsilon h^{\beta_1}.
\end{equation}

For the last term, it is easy to see that
\begin{equation}\label{eq:term3}
\left(\frac{\epsilon}{1-\epsilon}g(0)\right)^2\lesssim \epsilon^2(m\wedge 1)^2,
\end{equation}
since $g(0)\leq m$ by the assumption and $g(0)\lesssim 1$ by the fact that $g\in\mathcal{P}(\beta_1,L_1)$.

With the relation $\mathbb{E}(A_1+A_2+A_3)^2\lesssim \mathbb{E}A_1^2+ \mathbb{E}A_2^2 + \mathbb{E}A_3^2$ and the three bounds in (\ref{eq:term1}), (\ref{eq:term2}) and (\ref{eq:term3}), we conclude the proof by the specific choice of $h=n^{-\frac{1}{2\beta_0+1}}\wedge n^{-\frac{1}{2\beta_1+1}}\epsilon^{-\frac{2}{2\beta_1+1}}$.
\end{proof}

\begin{proof}[Proof of Theorem \ref{thm:fixedbandwidth}]
The error decomposes as
\begin{align*}
\wh f(0)-f(0)=(\wh f(0)-\mathbb{E}\wh f(0))+(\mathbb{E}\wh f(0)-(1-\epsilon)f(0)-\epsilon g(0))+\epsilon(g(0)-f(0)).
\end{align*}
Using the same argument that leads to (\ref{eq:term1}), we have $\mathbb{E}(\wh f(0)-\mathbb{E}\wh f(0))^2\lesssim \frac{1}{nh}$ for the variance term.
The bias term $(\mathbb{E}\wh f(0)-(1-\epsilon)f(0)-\epsilon g(0))$ can be further decomposed as
\begin{align*}
(1-\epsilon)\int \frac{1}{h}K\left(\frac{x}{h}\right)(f(x)-f(0))dx+\epsilon\int \frac{1}{h}K\left(\frac{x}{h}\right)(g(x)-g(0))dx.
\end{align*}
Therefore, the same argument that leads to (\ref{eq:term2}) also gives the bound
\begin{align*}
|\mathbb{E}\wh f(0)-(1-\epsilon)f(0)-\epsilon g(0)|\lesssim h^{\beta_0} + \epsilon h^{\beta_1}.
\end{align*}
For the last term, we have
$\epsilon|g(0)-f(0)|\lesssim\epsilon$. Combining the three bounds above, we have
$$\mathbb{E}\left(\wh f(0)-f(0)\right)^2\lesssim \frac{1}{nh}+h^{2\beta_0}+\epsilon^2.$$
Choose $h=n^{-\frac{1}{2\beta_0+1}}$, and the proof is complete.
\end{proof}

\subsection{Proof of Theorem \ref{thm:smoothlowerbound}}\label{sec:pf-lower-structure}

The proof of Theorem \ref{thm:smoothlowerbound} mainly relies on Le Cam's two-point argument. The method is summarized by the following lemma.
\begin{lemma}\label{lem:lowerbound}
Consider two distributions $P_{\theta_0}$ and $P_{\theta_1}$ whose parameters of interest are separated by $\Delta=|T_{\theta_0}-T_{\theta_1}|$. Assume
$\chi^2\left(P_{\theta_0},P_{\theta_1}\right)\leq \alpha$.
Then, we have
$$\inf_{\wh{T}}\sup_{\theta\in\{\theta_0,\theta_1\}}\mathbb{E}_{\theta}\left(\wh{T}-T_{\theta}\right)^2\geq \frac{1}{8}e^{-\alpha}\Delta^2.$$
\end{lemma}
We refer the readers to \cite{yu1997assouad} and \cite[Chapter 2.3]{tsybakov09} for rigorous proofs.
In the setting of Theorem \ref{thm:smoothlowerbound}, we need to find two pairs of density functions $(f,g)$ and  $(\wt f,\wt g)$ that satisfy $f,\wt f\in \mathcal{P}(\beta_0,L_0)$, $g,\wt g\in \mathcal{P}(\beta_1,L_1)$ and  $g(0)\vee \wt g(0)\leq m$. Since we are working with i.i.d. observations, it is sufficient to show that
\begin{align*}
\chi^2\left(p(\epsilon,\wt f,\wt g),p(\epsilon,f,g)\right)\lesssim n^{-1}.
\end{align*}
Then, Lemma \ref{lem:lowerbound} implies $\mathcal R(\epsilon,\beta_0,\beta_1,L_0,L_1,m)\gtrsim |f(0)-\wt f(0)|^2$.

The lower bound of Theorem \ref{thm:smoothlowerbound} contains three terms. We thus split the proof into three parts, and then combine the three arguments in the end.
\begin{lemma}\label{lem:term1}
We have
\begin{align*}
\mathcal R(\epsilon,\beta_0,\beta_1,L_0,L_1,m)\gtrsim n^{-\frac{2\beta_0}{2\beta_0+1}}.
\end{align*}
\end{lemma}
\begin{proof}
The proof uses a similar argument in \cite[Chapter 2.5]{tsybakov09}. Since we are dealing with a setting with contamination, we still give a proof to be self contained.
We define the following four functions,
\begin{align*}
g(x)&=\wt g(x)=c_{1}a(c_{1}x),\\
f(x)&=f_0(x),\\
\wt f(x)&=f_0(x)+c_{2}h^{\beta_0}b\left(\frac{x}{h}\right). 
\end{align*}
Here, we take $f_0$ as the density function of some normal distribution with mean zero so that $f_0\in\mathcal{P}(\beta_0,L_0/2)$. The functions $a(x)$ and $b(x)$ are given by Lemma \ref{lem:a} and Lemma \ref{lem:b}.
We first verify that for appropriate choices of $c_1,c_2$ and $h\leq 1$, the constructed functions are well-defined densities in the desired parameter spaces.
\begin{itemize}
  \item We have $f\in \mathcal{P}(\beta_0,L_0)$ by construction. Since $h\leq 1$, $b(x/h)$ is compactly supported on an area where $f_0$ is lower bounded by some positive constant. Thus, with a $c_2>0$ that is sufficiently small, $\wt f$ is nonnegative. The fact $\int \wt{f}=1$ can be derived from the property of $b$ in Lemma \ref{lem:b}. Hence, $\wt f\in \mathcal{P}(\beta_0,L_0)$ when $c_{2}$ is small enough.
  \item With a sufficiently small $c_1>0$, we have $g,\wt{g}\in \mathcal{P}(\beta_1,L_1)$.
  \item By $a(0)=0$ according to Lemma \ref{lem:a}, we get $|g(0)|\vee |\wt g(0)|\leq m$.
\end{itemize}
We use the notation $p=(1-\epsilon)f+\epsilon g$ and $q=(1-\epsilon)\wt{f}+\epsilon\wt{g}$. Note that $p$ can be lower bounded by a positive constant on the interval $[-1,1]$ according to its definition. Moreover, we have
$$p(x)-q(x)=-(1-\epsilon)c_{2}h^{\beta_0}b\left(\frac{x}{h}\right),$$
and the support of $b\left(\frac{x}{h}\right)$ is $[-h,h]\subset[-1,1]$. This leads to the bound
$$\chi^2(q,p)=\int_{-1}^1\frac{(p-q)^2}{p}\lesssim \int (p-q)^2 \asymp  h^{2\beta_0}\int b^2\left(\frac{x}{h}\right)\asymp {h}^{2\beta_0+1}.$$
In order that $n\chi^2(q,p)\lesssim 1$, we can choose $h=n^{-\frac{1}{2\beta_0+1}}$. This leads to
$$|f(0)-\wt f(0)|\asymp n^{-\frac{\beta_0}{2\beta_0+1}}.$$
Use Lemma \ref{lem:lowerbound}, and the proof is complete.
\end{proof}

\begin{lemma}\label{lem:term2}
We have
\begin{align*}
\mathcal R(\epsilon,\beta_0,\beta_1,L_0,L_1,m)\gtrsim \epsilon^2(1\wedge m)^2.
\end{align*}
\end{lemma}
\begin{proof}
By \cite{tsybakov09}, for any $p\in\mathcal{P}(\beta,L)$, there exists a constant $p_{max}$ such that $\sup_x|p(x)|\leq p_{max}$. Therefore, it is sufficient to consider $m$ that is bounded by some constant, say $m\leq 1$.
Consider the following four functions,
\begin{align*}
f(x)&=f_0(x),\\
\wt f(x) &=f_0(x)+c_1\frac{\epsilon}{1-\epsilon}mb(x),\\
g(x)&=c_2a(c_2x)+c_1mb(x),\\
\wt g(x)&=c_2a(c_2x).
\end{align*}
Here, we take $f_0$ as the density function of some normal distribution with mean zero so that $f_0\in\mathcal{P}(\beta_0,L_0/2)$. The functions $a(x)$ and $b(x)$ are given by Lemma \ref{lem:a} and Lemma \ref{lem:b}.
With appropriate choices of the constants $c_1,c_2>0$, $f,\wt{f},g,\wt{g}$ are well-defined density functions that belong to the desired function classes.
\begin{itemize}
  \item By Lemma \ref{lem:a}, 
  We have $f_0\in\mathcal{P}(\beta_0,L_0/2)\subset \mathcal{P}(\beta_0,L_0)$ by construction. Since $f_0$ is strictly positive on $[-1,1]$ and $b$ is compactly supported on $[-1,1]$, we have $\wt f\in \mathcal{P}(\beta_0,L_0)$ for some sufficiently small constant $c_1>0$ according to the properties of $b$ listed in Lemma \ref{lem:b}.
  \item By definition of $a$, we have $\wt g \in \mathcal{P}(\beta_1,L_1/2)$ for some sufficiently small $c_2>0$ according to Lemma \ref{lem:a}. Since $b(x)$ only takes negative values when $c_2a(c_2x)$ is lower bounded by a positive constant, $g$ is nonnegative and $g\in \mathcal{P}(\beta_1,L_1)$ when $c_1$ is small enough.
  \item We also have $|g(0)|\vee |\wt g(0)|\leq m$ for a sufficiently small $c_1$ because $a(0)=0$ and $|b(0)|$ is bounded by a constant according to Lemma \ref{lem:a} and Lemma \ref{lem:b}.
\end{itemize}
In summary, we have
$$(1-\epsilon)f+\epsilon g, (1-\epsilon)\wt{f}+\epsilon\wt{g}\in \mathcal{M}(\epsilon,\beta_0,\beta_1,L_0,L_1,m).$$
Moreover, according to our construction, we have
$$(1-\epsilon)f+\epsilon g= (1-\epsilon)\wt{f}+\epsilon\wt{g},$$
and
$$|f(0)-\wt{f}(0)|=c_1\frac{\epsilon}{1-\epsilon}m|b(0)|\gtrsim m\epsilon,$$
where we have used $|b(0)|\gtrsim 1$ by Lemma \ref{lem:b}.
Finally, using Lemma \ref{lem:lowerbound}, we obtain the desired lower bound result.
\end{proof}

\begin{lemma}\label{lem:term3}
Assume $\beta_1\leq \beta_0$ and $n\epsilon^2\geq 1$. Then, we have
\begin{align*}
\mathcal R(\epsilon,\beta_0,\beta_1,L_0,L_1,m)\gtrsim n^{-\frac{2\beta_1}{2\beta_1+1}}\epsilon^{\frac{2}{2\beta_1+1}}.
\end{align*}
\end{lemma}
\begin{proof}
Consider the following four functions,
\begin{align*}
f(x) &=f_0(x),\\
\wt f(x) &=f_0(x)+c_{2}\frac{\epsilon}{1-\epsilon}\left[h^{\beta_0}l\left(\frac{x}{h}\right)-h^{\beta_0}l\left(\frac{2(x-c_{4})}{h}\right)-h^{\beta_0}l\left(\frac{2(x+c_{4})}{h}\right)\right],\\
g(x)&=c_{1}a(c_{1}x)+c_{2}\left[h^{\beta_0}l\left(\frac{x}{h}\right)-h^{\beta_0}l\left(\frac{2(x-c_{4})}{h}\right)-h^{\beta_0}l\left(\frac{2(x+c_{4})}{h}\right)\right]-c_{3}\wt h^{\beta_1}b\left(\frac{x}{\wt h}\right),\\
\wt g(x)&=c_{1}a(c_{1}x).
\end{align*}
Since the proof relies on perturbing a density at a point where it is $0$, the verification of nonnegativity is more delicate, which motivates another tuning constant controlling the center of the negative part of the perturbation.
Here, we take $f_0$ as the density function of some normal distribution with mean zero so that $f_0\in\mathcal{P}(\beta_0,L_0/2)$. The functions $a(x)$ and $b(x)$ are given by Lemma \ref{lem:a} and Lemma \ref{lem:b}. The numbers $h$ and $\wt{h}$ are chosen so that the following equation is satisfied:
\begin{equation}\label{eq:relation}
c_{2}h^{\beta_0}l(0)=c_{3}\wt h^{\beta_1}b(0).
\end{equation}
Now, we verify that with appropriate choices of constants $c_1,c_2,c_3,c_4$, the constructed functions belong to the parameter spaces.
\begin{itemize}
  \item The functions $f$ and $\wt g$ are automatically density functions by definition. Note that we can choose a small  constant $c_4$ so that the negative perturbation $-h^{\beta_0}l\left(\frac{2(x-c_{4})}{h}\right)-h^{\beta_0}l\left(\frac{2(x+c_{4})}{h}\right)$ has a support in a region where both $f_0$ and $c_1a(c_1x)$ are bounded below by a positive constant. This immediately implies that $\wt{f}(x)\geq 0$ for all $x$ with a sufficiently small constant $c_2$. Similarly, the support of $-c_{3}\wt h^{\beta_1}b\left(\frac{x}{\wt h}\right)$ is $[-\wt{h},\wt{h}]$, which is contained in a region where $c_1a(c_1x)$ is bounded below by a positive constant for a sufficiently small $\wt{h}$. Therefore, $g(x)\geq 0$ for all $x$ with a sufficiently small constant $c_3$. We also note that $\int \wt{f}=\int g=1$ according to the definitions.
  \item When $c_{1},c_{2},c_{3}$ are chosen small enough, we have $f,\wt f\in \Sigma(\beta_0,L_0)$ and $g,\wt g \in \Sigma(\beta_1,L_1)$. Here $g\in \Sigma(\beta_1,L_1)$ is a consequence of the assumption that $\beta_1\leq\beta_0$.
  \item Finally, we have $l(2c_4/h)=l(-2c_4/h)=0$ for a sufficiently small $h$. This implies $g(0)=\wt g(0)=0$ because of (\ref{eq:relation}). Therefore, $|g(0)\vee \wt g(0)|\leq m$.
\end{itemize}
In summary, we have
$$(1-\epsilon)f+\epsilon g, (1-\epsilon)\wt{f}+\epsilon\wt{g}\in \mathcal{M}(\epsilon,\beta_0,\beta_1,L_0,L_1,m).$$
Besides the properties listed above, we also note that both $f$ and $g$ can be bounded from below by some positive constant on the interval $[-1,1]$, if the constants $c_2,c_3$ are sufficiently small. This implies that the density $(1-\epsilon)f+\epsilon g$ is lower bounded by some positive constant on the interval $[-1,1]$.

Now, according to the above construction, for $p=(1-\epsilon)f+\epsilon g$ and $q=(1-\epsilon)\wt{f}+\epsilon\wt{g}$, we have
\begin{align*}
p(x)-q(x)=-\epsilon c_{3}\wt h^{\beta_1}b\left(\frac{x}{\wt h}\right).
\end{align*}
Given that the support of $b\left(\frac{x}{\wt h}\right)$ is within $[-\wt{h},\wt{h}]\subset[-1,1]$ with a sufficiently small $\wt{h}$, we have
$$\chi^2(q,p)=\int_{-1}^1\frac{(p-q)^2}{p}\lesssim \int (p-q)^2 \asymp \epsilon^2 \wt{h}^{2\beta_1}\int b^2\left(\frac{x}{\wt h}\right)\asymp \epsilon^2\wt{h}^{2\beta_1+1}.$$
In order that $n\chi^2(q,p)\lesssim 1$, it is sufficient to choose $\wt{h}\asymp\left(n\epsilon^2\right)^{-\frac{1}{2\beta_1+1}}$. The condition $n\epsilon^2\geq 1$ implies that $\wt{h}$ can be picked sufficiently small.
Moreover, with the relation (\ref{eq:relation}), we have
$$|f(0)-\wt{f}(0)|=c_2\frac{\epsilon}{1-\epsilon}h^{\beta_0}l(0)\asymp \epsilon h^{\beta_0}\asymp \epsilon \wt{h}^{\beta_1}\asymp\epsilon^{\frac{1}{2\beta_1+1}}n^{-\frac{\beta_1}{2\beta_1+1}}.$$
Finally, using Lemma \ref{lem:lowerbound}, we obtain the desired lower bound result.
\end{proof}

We combine the results of Lemma \ref{lem:term1}, Lemma \ref{lem:term2} and Lemma \ref{lem:term3}.
\begin{proof}[Proof of Theorem \ref{thm:smoothlowerbound}]
In order that the third term $n^{-\frac{2\beta_1}{2\beta_1+1}}\epsilon^{\frac{2}{2\beta_1+1}}$ dominates the other two, it is necessary that $\epsilon^2\geq n^{\frac{2\beta_1-2\beta_0}{2\beta_0+1}}$. This implies both $\beta_1\leq \beta_0$ and $n\epsilon^2\geq 1$. By Lemma \ref{lem:term3}, we have
$$\mathcal R(\epsilon,\beta_0,\beta_1,L_0,L_1,m)\gtrsim n^{-\frac{2\beta_1}{2\beta_1+1}}\epsilon^{\frac{2}{2\beta_1+1}}.$$
When the first or the second term dominate, we use Lemma \ref{lem:term1} and Lemma \ref{lem:term2}, and obtain
$$\mathcal R(\epsilon,\beta_0,\beta_1,L_0,L_1,m)\gtrsim [n^{-\frac{2\beta_0}{2\beta_0+1}}]\vee[\epsilon^2(1\wedge m)^2].$$
Hence, the proof is complete.
\end{proof}

\subsection{Proofs of Theorem \ref{thm:smadapt1} and Theorem \ref{thm:smadapt2}}\label{sec:pf-cri}

The proofs of both theorems rely on the following constrained risk inequality by \cite{brown1996constrained}.
\begin{lemma}\label{lem:crineq}
Consider two distributions $P_{\theta_0}$ and $P_{\theta_1}$ whose parameters of interest are separated by $\Delta=|T_{\theta_0}-T_{\theta_1}|$. For any estimator $\wh{T}$, assume
$$
\mathbb{E}_{\theta_0}(\wh T-T_{\theta_0})^2\leq \delta^2.
$$
Then, whenver $\delta I\leq \Delta$, we have
$$
\mathbb{E}_{\theta_1}(\wh T-T_{\theta_1})^2\geq (\Delta-\delta I)^2,
$$
where $I=\sqrt{\int \frac{dP_{\theta_1}^2}{dP_{\theta_0}}}$.
\end{lemma}

\begin{proof}[Proof of Theorem \ref{thm:smadapt1}]
We consider the following four functions,
\begin{align*}
f&=f_0,\\
g&=c_1a(c_1x),\\
\wt f&= \frac{1-\epsilon}{1-\wt\epsilon}f_0+\frac{\epsilon-\wt\epsilon}{1-\wt\epsilon}c_1a(c_1x),\\
\wt g&=c_1a(c_1x).
\end{align*}
Here, we take $f_0$ as the density function of some normal distribution with mean zero so that $f_0\in\mathcal{P}(\beta_0,L_0/2)$. The function $a(\cdot)$ is given by Lemma \ref{lem:a}. The constant $c_1$ is sufficiently small so that $c_1a(c_1x)$ belongs to both $\mathcal{P}(\beta_0,L_0/2)$ and $\mathcal{P}(\beta_1,L_1/2)$. Now it is easy to check that $f,\wt f\in \mathcal{P}(\beta_0,L_0)$, $g,\wt g\in \mathcal{P}(\beta_1,L_1)$ and $g(0)\vee\wt g(0)=0\leq m$, so that the constructed functions are well-defined densities in the parameter spaces.

It is easy to check that
$$(1-\epsilon)f+\epsilon g=(1-\wt\epsilon)\wt{f}+\wt\epsilon \wt{g}.$$
This implies $\int q^2/p=1$ for $p=(1-\epsilon)f+\epsilon g$ and $q=(1-\wt\epsilon)\wt{f}+\wt\epsilon \wt{g}$. We also have
$$\left|f(0)-\wt{f}(0)\right|=\frac{\epsilon-\wt\epsilon}{1-\wt\epsilon}f(0).$$
According to Lemma \ref{lem:crineq}, suppose there is an estimator $\wh{f}(0)$ that satisfies $\mathbb{E}_{p^n}\left(\wh f(0)-f(0)\right)^2\leq C\wt\epsilon^2$, we must have
$$\mathbb{E}_{q^n}(\wh{f}(0)-\wt{f}(0))^2\geq \left(\frac{\epsilon-\wt\epsilon}{1-\wt\epsilon}f(0)-C^{1/2}\wt\epsilon\right)^2.$$
Therefore, there exists a constant $C'>0$, such that for $\epsilon\geq C'\wt\epsilon$, $\mathbb{E}_{q^n}(\wh{f}(0)-\wt{f}(0))^2\gtrsim {\epsilon}^2$, and the proof is complete.
\end{proof}

\begin{proof}[Proof of Theorem \ref{thm:smadapt2}]
We construct the following four functions
\begin{align*}
\wt f(x) &=f_0(x),\\
f(x) &=f_0(x)-c_{2}\frac{\epsilon}{1-\epsilon}\left[h^{\beta_0}l\left(\frac{x}{h}\right)-h^{\beta_0}l\left(\frac{2(x-c_{4})}{h}\right)-h^{\beta_0}l\left(\frac{2(x+c_{4})}{h}\right)\right],\\
\wt g(x)&=c_{1}a(c_{1}x),\\
{g}(x)&=c_{1}a(c_{1}x)+c_{2}\left[h^{\beta_0}l\left(\frac{x}{h}\right)-h^{\beta_0}l\left(\frac{2(x-c_{4})}{h}\right)-h^{\beta_0}l\left(\frac{2(x+c_{4})}{h}\right)\right]-c_{3}\wt h^{\beta_1}b\left(\frac{x}{\wt h}\right).
\end{align*}
The construction is similar to that in the proof of Lemma \ref{lem:term3}. The difference is that the perturbation is now put on both ${f}$ and ${g}$. Here, we take $f_0$ as the density function of some normal distribution with mean zero so that $f_0\in\mathcal{P}(\beta_0,L_0/2)$. The functions $a(x)$ and $b(x)$ are given by Lemma \ref{lem:a} and Lemma \ref{lem:b}. The numbers $h$ and $\wt{h}$ are chosen so that the following equation is satisfied:
\begin{equation}\label{eq:relation2}
c_{2}h^{{\beta}_0}l(0)=c_{3}\wt h^{{\beta}_1}b(0).
\end{equation}
Similar to the argument used in Lemma \ref{lem:term3}, it is not hard to check that with appropriate choices of the constants $c_1,c_2,c_3$, we have $\wt f\in\mathcal{P}(\wt\beta_0,\wt L_0)$, $\wt g\in\mathcal{P}(\wt\beta_1,\wt L_1)$, ${f}\in\mathcal{P}({\beta}_0,{L}_0)$ and ${g}\in\mathcal{P}({\beta}_1,{L}_1)$, given that $\wt\beta_0\geq {\beta}_0\geq{\beta}_1$ and $\wt{\beta}_1>\beta_1$. The numbers $h$ and $\wt{h}$ are both required to be sufficiently small. We also have $g(0)=\wt{g}(0)=0$ according to the definition with an appropriate choice of $c_4$. Then, the constructed functions are well-defined densities in the parameter spaces.

With the notation $p=(1-\epsilon)f+\epsilon g$ and $q=(1-\epsilon)\wt{f}+\epsilon\wt{g}$, we check the quantities in Lemma \ref{lem:crineq}. Note that
$$|p(x)-q(x)|=c_{3}\epsilon\wt h^{\beta_1}b\left(\frac{x}{\wt h}\right).$$
With a similar argument in the proof of Lemma \ref{lem:term3}, the function $b\left(\frac{x}{\wt h}\right)$ is supported within $[-\wt{h},\wt{h}]\subset[-1,1]$, and $p(x)$ is lower bounded by some constant uniformly over $x\in[-1,1]$. This implies,
$$I=\left(\int\frac{q^2}{p}\right)^{\frac{n}{2}}=\left(1+\int \frac{(q-p)^2}{p}\right)^{\frac{n}{2}}\leq \exp\left(\frac{C_1n}{2}\int (p-q)^2\right)\leq \exp\left(C_1'n\epsilon^2\wt{h}^{2{\beta}_1+1}\right).$$
Moreover, we also have
$$\Delta=|f(0)-\wt{f}(0)|=c_2\frac{\epsilon}{1-\epsilon}h^{{\beta}_0}l(0),$$
and
$$\delta=C^{1/2}\left(\frac{n}{\log n}\right)^{-\frac{\wt\beta_1}{2\wt\beta_1+1}}\epsilon^{\frac{1}{2\wt\beta_1+1}}.$$
In order that $I\leq \left(\frac{n\epsilon^2}{\log n}\right)^c$ for some sufficiently small constant $c>0$, we can choose $\wt{h}\asymp \left(\frac{n\epsilon^2}{\log n}\right)^{-\frac{1}{2{\beta}_1+1}}$, which is always possible with the condition $n\epsilon^2\geq (\log n)^2$. According to the relation (\ref{eq:relation2}), we have $\Delta\asymp \epsilon^{\frac{1}{2{\beta}_1+1}}\left(\frac{n}{\log n}\right)^{-\frac{{\beta}_1}{2{\beta}_1+1}}$.
Plugging these quantities into the constrained risk inequality in Lemma \ref{lem:crineq} and using $\beta_1<\wt{\beta}_1$, we get the desired lower bound.
\end{proof}

\subsection{Proofs of Theorem \ref{thm:lepski2} and Theorem \ref{thm:lepski}}

The proofs of the two theorems are similar. Thus, we give a detailed proof of Theorem \ref{thm:lepski} first, and then sketch the proof of Theorem \ref{thm:lepski2}.

\begin{proof}[Proof of Theorem \ref{thm:lepski}]
For every bandwidth $h$, the error decomposes as
\begin{equation}\label{eq:riskdecomposition}
\wh f_h(0)-f(0)=(\wh f_h(0)-\mathbb{E}\wh f_h(0))+(\mathbb{E}\wh f_h(0)-(1-\epsilon)f(0)-\epsilon g(0))+\epsilon(g(0)-f(0)),
\end{equation}
where the three terms correspond to a stochastic part that depends on $h$, a deterministic part that depends on $h$, and a deterministic part that does not depend on $h$. With the same argument in the proof of Theorem \ref{thm:fixedbandwidth}, we have
$$\mathbb{E}(\wh f(0)-\mathbb{E}\wh f(0))^2\lesssim \frac{1}{nh},$$
$$|\mathbb{E}\wh f(0)-(1-\epsilon)f(0)-\epsilon g(0)|\lesssim h^{\beta_0} + \epsilon h^{\beta_1},$$
and
$$\epsilon|g(0)-f(0)|\lesssim\epsilon.$$
Define the oracle bandwidth $h_*$ to be the largest $h\in \mathcal{H}$ such that
\begin{align*}
h^{\beta_0} + \epsilon h^{\beta_1}\leq c\sqrt{\frac{\log n}{nh}},
\end{align*}
where the constant $c>0$ will be determined later.
Then it is easy to see that $h_*$ satisfies
\begin{equation}\label{eq:optimaltradeoff}
 c'\sqrt{\frac{\log n}{nh_*}}\leq h_*^{\beta_0}+\epsilon h_*^{\beta_1} \leq c\sqrt{\frac{\log n}{nh_*}},
\end{equation}
for some constant $c'$ that only depends on $c$.

We proceed to prove that $\wh h\geq h_*$ with high probability. By the definition of $\wh{h}$, we have
\begin{align*}
\mathbb{P}(\wh h< h_*)&\leq \mathbb{P}\left(\exists l\leq h_*\text{ and }l\in\mathcal{H} \mbox{ s.t. } |\wh f_{h_*}(0)-\wh f_l(0)|> c_1\sqrt{\frac{\log n}{nl}}\right)\\
&\leq \sum_{l\leq h_*,l\in\mathcal{H}}\mathbb{P}\left(|\wh f_{h_*}(0)-\wh f_l(0)|> c_1\sqrt{\frac{\log n}{nl}}\right).
\end{align*}
We derive a bound for $\mathbb{P}\left(|\wh f_{h_*}(0)-\wh f_l(0)|> c_1\sqrt{\frac{\log n}{nl}}\right)$ for each $l\leq h_*$ and $l\in\mathcal{H}$.
Due to the error decomposition (\ref{eq:riskdecomposition}), we have:
\begin{align*}
|\wh f_{h_*}(0)-\wh f_l(0)|\leq C(h_*^{\beta_0}+\epsilon h_*^{\beta_1})+|\wh f_{h_*}(0)-\mathbb{E}\wh f_{h_*}(0)|+|\wh f_l(0)-\mathbb{E}\wh f_l(0)|,
\end{align*}
for some constant $C>0$.
By (\ref{eq:optimaltradeoff}),
the bias term can be controlled as
$$C(h_*^{\beta_0}+\epsilon h_*^{\beta_1})\leq  C\times c\sqrt{\frac{\log n}{nh_*}}\leq \frac{c_1}{2}\sqrt{\frac{\log n}{nl}},$$
for a sufficiently small $c>0$.
Thus, we have
\begin{align*}
\mathbb{P}(\wh h< h_*)&\leq \sum_{l\leq h_*,l\in\mathcal{H}}\mathbb{P}\left(|\wh f_{h_*}(0)-\mathbb{E}\wh f_{h_*}(0)|+|\wh f_l(0)-\mathbb{E}\wh f_l(0)|\geq \frac{c_1}{2}\sqrt{\frac{\log n}{nl}}\right) \\
& \leq \sum_{l\leq h_*,l\in\mathcal{H}}\mathbb{P}\left(|\wh f_{h_*}(0)-\mathbb{E}\wh f_{h_*}(0)|\geq \frac{c_1}{4}\sqrt{\frac{\log n}{nh^*}}\right) \\
& + \sum_{l\leq h_*,l\in\mathcal{H}}\mathbb{P}\left(|\wh f_{l}(0)-\mathbb{E}\wh f_{l}(0)|\geq \frac{c_1}{4}\sqrt{\frac{\log n}{nl}}\right).
\end{align*}
For any $l\leq h_*$ and $l\in\mathcal{H}$, we use Bernstein's inequality, and get
\begin{eqnarray*}
&& \mathbb{P}\left(|\wh f_{l}(0)-\mathbb{E}\wh f_{l}(0)|\geq t\right) \\
&\leq& \mathbb{P}\left(\left|\frac{1}{n}\sum_{i=1}^nl^{-1}K(X_i/l)-\mathbb{E}l^{-1}K(X/l)\right|\geq t\right) \\
&\leq& 2\exp\left(-\frac{nt^2/2}{\sigma^2+Mt/3}\right),
\end{eqnarray*}
where we choose $t=\frac{c_1}{4}\sqrt{\frac{\log n}{nl}}$, and $\sigma^2$ and $M$ have bounds
$$\sigma^2\leq \mathbb{E}l^{-2}K^2(X/l)\lesssim l^{-1}\text{ and }M\lesssim l^{-1}.$$
This implies the bound
\begin{equation}
\mathbb{P}\left(|\wh f_{l}(0)-\mathbb{E}\wh f_{l}(0)|\geq \frac{c_1}{4}\sqrt{\frac{\log n}{nl}}\right)\leq 2\exp\left(-C'\log n\right),\label{eq:haoyang-nb}
\end{equation}
where the constant $C'>0$ can be arbitrarily large given a sufficiently large $c_1>0$.
For example, we set a large enough $c_1>0$ so that $C'=3$. This gives
$$\mathbb{P}(\wh h< h_*)\leq 4|\mathcal{H}|n^{-3}\lesssim n^{-3}\log n.$$

Now, on the event $\{\wh h\geq h_*\}$, the risk decomposes as
\begin{align*}
|\wh f_{\wh h}(0)-f(0)|\leq |\wh f_{\wh h}(0)-\wh f_{h_*}(0)|+|\wh f_{h_*}(0)-f(0)|.
\end{align*}
Due to the definition of $\wh h$, the first term satisfies
\begin{equation}\label{eq:oraclebandwidth1}
|\wh f_{\wh h}(0)-\wh f_{h_*}(0)|\leq c_1\sqrt{\frac{\log n}{nh_*}}.
\end{equation}
For the second term, the error decomposition and the relation (\ref{eq:optimaltradeoff}) implies
$$\mathbb{E}|\wh f_{h_*}(0)-f(0)|^2\lesssim \frac{\log n}{nh_*}+\epsilon^2.$$
Therefore, we have
\begin{eqnarray*}
&& \mathbb{E}(\wh f_{\wh h}(0)-f(0))^2 \\
&\leq& \mathbb{E}((\wh f_{\wh h}(0)-f(0))^2:\wh h\geq h_*)+\mathbb{E}((\wh f_{\wh h}(0)-f(0))^2:\wh h< h_*)\\
&\leq& 2\mathbb{E}((\wh f_{\wh h}(0)-\wh f_{h_*}(0))^2:\wh h\geq h_*)+2\mathbb{E}((\wh f_{h_*}(0)-f(0))^2:\wh h\geq h_*)+O\left(n^2\mathbb{P}(\wh h< h_*)\right)\\
&\lesssim& \frac{\log n}{nh_*} +\epsilon^2 + \frac{\log n}{n} \\
&\lesssim& \left(\frac{\log n}{n}\right)^{\frac{2\beta_0}{2\beta_0+1}}+\epsilon^2.
\end{eqnarray*}
The last inequality above is by realizing that $h_*\asymp \left(\frac{n}{\log n}\right)^{-\frac{1}{2\beta_0+1}}$ from the relation (\ref{eq:optimaltradeoff}). The proof is complete.
\end{proof}

\begin{proof}[Proof of Theorem \ref{thm:lepski2}]
The proof for Theorem \ref{thm:lepski2} follows the same argument as that of Theorem \ref{thm:lepski}. The only difference lies in the normalization, which leads to the error decomposition
$$\wh{f}_h(0)-f(0)=(\wh{f}_h(0)-\mathbb{E}\wh{f}_h(0))+\left(\mathbb{E}\wh{f}_h(0)-f(0)-\frac{\epsilon}{1-\epsilon}g(0)\right)+\frac{\epsilon}{1-\epsilon}g(0).$$
The rest of the details are the same and is omitted.
\end{proof}


\subsection{Proof of Theorem \ref{thm:minimax-arb}}\label{sec:pf-miss}

We split the proof into upper and lower bounds. We first prove the following upper bound.
\begin{thm}
For the estimator $\wh{f}(0)=\wh{f}_h(0)$ with some $K\in \mathcal{K}_{\left \lfloor{\beta_0}\right \rfloor}(L)$ and $h=n^{-\frac{1}{2\beta_0+1}}\vee \epsilon^{\frac{1}{\beta_0+1}}$, we have
$$\sup_{p(\epsilon,f,g)\in \mathcal M(\epsilon,\beta_0,L_0)}\mathbb{E}_{p^n}\left(\wh f(0)-f(0)\right)^2\lesssim [n^{-\frac{2\beta_0}{2\beta_0+1}}]\vee[\epsilon^{\frac{2\beta_0}{\beta_0+1}}].$$
\end{thm}
\begin{proof}
Decompose the error as
\begin{align*}
\wh{f}_h(0)-f(0)=(\wh{f}_h(0)-\mathbb{E}\wh{f}_h(0))+(\mathbb{E}_h\wh{f}(0)-f(0)),
\end{align*}
where the first term is the stochastic error and the second term is the bias.
For the first term, we have
\begin{align*}
\mathbb{E}(\wh{f}_h(0)-\mathbb{E}\wh{f}_h(0))^2=\frac{1}{n}\Var\left(\frac{1}{h}K\left(\frac{X}{h}\right)\right),
\end{align*}
and
\begin{eqnarray*}
\Var\left(\frac{1}{h}K\left(\frac{X}{h}\right)\right) &\leq& (1-\epsilon)\int\frac{1}{h^2}K^2\left(\frac{x}{h}\right)f(x)dx + \epsilon\int\frac{1}{h^2}K^2\left(\frac{x}{h}\right)dG(x) \\
&\lesssim& \frac{1}{h}\int \frac{1}{h}K^2\left(\frac{x}{h}\right)dx+\frac{\epsilon}{h^2}\int dG(x) \\
&\lesssim& \frac{1}{h} + \frac{\epsilon}{h^2}.
\end{eqnarray*}
Therefore, we have
\begin{equation}\label{eq:term1new}
\mathbb{E}(\wh{f}_h(0)-\mathbb{E}\wh{f}_h(0))^2\lesssim\frac{1}{nh}+\frac{\epsilon}{nh^2}.
\end{equation}
For the bias term, we have
\begin{align*}
\mathbb{E} \wh{f}_h(0)-f(0)=(1-\epsilon)\int \frac{1}{h}K\left(\frac{x}{h}\right)(f(x)-f(0))dx+\epsilon\int  \frac{1}{h}K\left(\frac{x}{h}\right)dG(x)-\epsilon f(0),
\end{align*}
where the first term has bound
$$\left|\int \frac{1}{h}K\left(\frac{x}{h}\right)(f(x)-f(0))dx\right|\lesssim h^{\beta_0},$$
by \cite[Chapter 1.2]{tsybakov09}, and the next two terms can be bounded as
$$\left|\epsilon\int  \frac{1}{h}K\left(\frac{x}{h}\right)dG(x)-\epsilon f(0)\right|\lesssim \frac{\epsilon}{h}\int dG(x) +\epsilon f(0)\lesssim \frac{\epsilon}{h}.$$
Therefore, we have
\begin{equation}\label{eq:term2new}
|\mathbb{E}\wh{f}_h(0)-f(0)|\lesssim h^{\beta_0}+\frac{\epsilon}{h}.
\end{equation}
Combine the two bounds (\ref{eq:term1new}) and (\ref{eq:term2new}), choose $h=n^{-\frac{1}{2\beta_0+1}}\vee \epsilon^{\frac{1}{\beta_0+1}}$, and then we complete the proof.
\end{proof}

Now we state the lower bound.
\begin{thm}\label{thm:arbitrarylowerbound}
We have
$$
{\mathcal R}(\epsilon,\beta_0,L_0)\gtrsim [n^{-\frac{2\beta_0}{2\beta_0+1}}]\vee[\epsilon^{\frac{2\beta_0}{\beta_0+1}}].
$$
\end{thm}
Before proving this theorem, we need the following lemma.
\begin{lemma}\label{lem:dif}
A function $d(x)$ can be written as the difference of two density functions if and only if
$$\int d=0\quad\text{and}\quad \int|d|\leq 2.$$
\end{lemma}
\begin{proof}
The ``only if" part is obvious. Now assume the two conditions hold, and then for any density function $f$, we have the following decomposition for $d$,
\begin{align*}
d=\left[d_++\left(1-\frac{1}{2}\int |d|\right)f\right]-\left[d_-+\left(1-\frac{1}{2}\int |d|\right)f\right],
\end{align*}
where $d_+$ and $d_-$ are the positive and negative parts of the function. The first condition implies $\int d_+=\int d_-=\frac{1}{2}\int |d|$. Thus,
\begin{align*}
\int \left[d_++\left(1-\frac{1}{2}\int |d|\right)f\right]=\int\left[d_-+\left(1-\frac{1}{2}\int |d|\right)f\right]=1.
\end{align*}
The second condition guarantees that both $d_++\left(1-\frac{1}{2}\int |d|\right)f$ and $d_-+\left(1-\frac{1}{2}\int |d|\right)f$ are nonnegative. Thus, the proof is complete.
\end{proof}

\begin{proof}[Proof of Theorem \ref{thm:arbitrarylowerbound}]
For the lower bound of the first term $n^{-\frac{2\beta_0}{2\beta_0+1}}$, see the proof of Lemma \ref{lem:term1}. We give a proof for the second term. Consider the following two functions
\begin{align*}
f&=f_0,\\
\wt f &=f_0+ch^{\beta_0}b\left(\frac{x}{h}\right).
\end{align*}
Here, we take $f_0$ as the density function of some normal distribution with mean zero so that $f_0\in\mathcal{P}(\beta_0,L_0/2)$. The function $b$ is defined in Lemma \ref{lem:b}.
The constant $c$ is chosen small enough so that $f\in \mathcal{P}(\beta_0,L_0)$. In order that there exist $g$ and $\wt g$ so that
\begin{align*}
(1-\epsilon)f+\epsilon g=(1-\epsilon)\wt f+\epsilon\wt g,
\end{align*}
it suffices to verify the existence of densities $g$ and $\wt g$ such that
\begin{align*}
g(x)-\wt g(x)=\frac{1-\epsilon}{\epsilon}\left(\wt{f}(x)-f(x)\right)=c\frac{1-\epsilon}{\epsilon}h^{\beta_0}b\left(\frac{x}{h}\right).
\end{align*}
By Lemma \ref{lem:dif}, it further suffices to verify the condition
\begin{align*}
c\frac{1-\epsilon}{\epsilon}\int h^{\beta_0}\left|b\left(\frac{x}{h}\right)\right|dx\leq 2,
\end{align*}
and this is guaranteed by taking some $h\asymp \epsilon^{\frac{1}{\beta_0+1}}$. Now we have $g$ and $\wt{g}$ such that $(1-\epsilon)f+\epsilon g=(1-\epsilon)\wt f+\epsilon\wt g$ holds. Moreover,
$$|f(0)-\wt{f}(0)|=ch^{\beta_0}b(0)\asymp \epsilon^{\frac{\beta_0}{\beta_0+1}}.$$
Apply Lemma \ref{lem:lowerbound}, and the proof is complete.
\end{proof}

\subsection{Proofs of Theorem \ref{thm:lepski3} and Theorem \ref{thm:arbitrarylepski}}

We first prove Theorem \ref{thm:arbitrarylepski}. Then, the proof of Theorem \ref{thm:lepski3} will be sketched using arguments in the proofs of Theorem \ref{thm:lepski} and Theorem \ref{thm:arbitrarylepski}.
\begin{proof}[Proof of Theorem \ref{thm:arbitrarylepski}]
We consider observations $X_1,...,X_n$. We assume that $X_1,...,X_a$ are generated from the density $f$ with some integer $a$, and the remaining observations $X_{a+1},...,X_n$ are generated from contamination. The number $a$ follows $\text{Binomial}(n,1-\epsilon)$. This is without loss of generality, because the definition of $\wh{f}$ does not depend on the order of the data $X_1,...,X_n$. Apply Bernstein's inequality, and we get
$$\mathbb{P}\left(\frac{n-a}{n}\geq 2\epsilon\right)\leq \exp\left(-\frac{3}{8}n\epsilon\right).$$
From now on, we assume that $\epsilon\geq \frac{8\log n}{n}$, so that $\frac{n-a}{n}\leq 2\epsilon$ with probability at least $1-n^{-3}$. The case $\epsilon<  \frac{8\log n}{n}$ will be considered in the end of the proof. Moreover, the following analysis conditions on the event $\left\{\frac{n-a}{n}\leq 2\epsilon\right\}$, and we use $\bar{\mathbb{P}}$ and $\bar{\mathbb{E}}$ to denote probability and expectation conditioning on the random variable $a$.

We start by the following error decomposition,
\begin{eqnarray*}
\wh f_h(0)-f(0) &=& \frac{1}{n}\sum_{i=1}^a\left(h^{-1}K(X_i/h)-\mathbb{E}_{X\sim f}h^{-1}K(X/h)\right) \\
&& + \frac{a}{n}\left(\mathbb{E}_{X\sim f}h^{-1}K(X/h) - f(0)\right) \\
&& + \frac{1}{n}\sum_{i=a+1}^n h^{-1}K(X_i/h) - \frac{n-a}{n}f(0).
\end{eqnarray*}
With similar arguments used in the proof of Theorem \ref{thm:upperbound}, we have
$$\bar{\mathbb{E}}\left(\frac{1}{n}\sum_{i=1}^a\left(h^{-1}K(X_i/h)-\mathbb{E}_{X\sim f}h^{-1}K(X/h)\right)\right)^2\lesssim \frac{1}{nh},$$
and
$$\left|\mathbb{E}_{X\sim f}h^{-1}K(X/h) - f(0)\right|\lesssim h^{\beta_0}.$$
Moreover, $\frac{n-a}{n}\leq 2\epsilon$ implies that
$$\frac{1}{n}\sum_{i=a+1}^n h^{-1}K(X_i/h)\lesssim \frac{\epsilon}{h},$$
and $\frac{n-a}{n}f(0)\lesssim \epsilon$. These bounds motivate us to define an oracle bandwidth $h_*$ that is the smallest $h\in\mathcal{H}$ such that
\begin{align*}
\frac{\epsilon}{h}+ \sqrt{\frac{\log n}{nh}}\leq h^{\wt{\beta}_0}.
\end{align*}
Then it is obvious that $h_*$ satisfies
\begin{equation}\label{eq:hstarstar}
 ch_*^{\wt{\beta}_0}\leq\frac{\epsilon}{h_*}+ \sqrt{\frac{\log n}{nh_*}}\leq h_*^{\wt{\beta}_0},
\end{equation}
with some constant $c>0$. Now we prove that $\wh h\leq h_*$ holds with high probability. According to the definition of $\wh{h}$, we have
\begin{eqnarray*}
\bar{\mathbb{P}}(\wh{h}>h_*) &\leq& \bar{\mathbb{P}}\left(\exists l\geq h_*\text{ and }l\in \mathcal{H}\text{ s.t. } |\wh f_{h_*}(0)-\wh f_l(0) |\geq c_1l^{\wt \beta_0}\right) \\
&\leq&\sum_{l\geq h_*,l\in\mathcal{H}}\bar{\mathbb{P}}\left(|\wh f_{h_*}(0)-\wh f_l(0) |\geq c_1l^{\wt \beta_0}\right).
\end{eqnarray*}
By the risk decomposition, for $l\geq h_*$ and $l\in\mathcal{H}$, the difference $|\wh f_{h^*}(0)-\wh f_l(0) |$ is bounded as
\begin{eqnarray*}
|\wh f_{h_*}(0)-\wh f_l(0)| &\leq& \left|\frac{1}{n}\sum_{i=1}^a\left(h_*^{-1}K(X_i/h_*)-\mathbb{E}_{X\sim f}h_*^{-1}K(X/h_*)\right)\right| \\
&& + \left|\frac{1}{n}\sum_{i=1}^a\left(l^{-1}K(X_i/l)-\mathbb{E}_{X\sim f}l^{-1}K(X/l)\right)\right| \\
&& + C\left(\frac{\epsilon}{h_*}+l^{\beta_0}\right),
\end{eqnarray*}
for some constant $C>0$. According to (\ref{eq:hstarstar}) and the condition $\wt{\beta}_0\leq \beta_0$, we have
$$C\left(\frac{\epsilon}{h_*}+l^{\beta_0}\right)\leq C\left(h_*^{\wt{\beta}_0}+l^{\wt{\beta}_0}\right)\leq \frac{c_1}{4}h_*^{\wt{\beta}_0}+\frac{c_1}{4}l^{\wt{\beta}_0},$$
where the last inequality holds for a sufficiently large $c_1$.
Thus, we have the bound
\begin{eqnarray*}
\bar{\mathbb{P}}(\wh{h}>h_*) &\leq& \sum_{l\geq h_*,l\in\mathcal{H}}\bar{\mathbb{P}}\left(|\wh f_{h_*}(0)-\wh f_l(0) |\geq \frac{c_1}{2}l^{\wt \beta_0}+\frac{c_1}{2}h_*^{\wt \beta_0}\right) \\
&\leq& \sum_{l\geq h_*,l\in\mathcal{H}}\bar{\mathbb{P}}\left( \left|\frac{1}{n}\sum_{i=1}^a\left(h_*^{-1}K(X_i/h_*)-\mathbb{E}_{X\sim f}h_*^{-1}K(X/h_*)\right)\right|\geq \frac{c_1}{4}h_*^{\wt{\beta}_0}\right) \\
&& + \sum_{l\geq h_*,l\in\mathcal{H}}\bar{\mathbb{P}}\left( \left|\frac{1}{n}\sum_{i=1}^a\left(l^{-1}K(X_i/l)-\mathbb{E}_{X\sim f}l^{-1}K(X/l)\right)\right|\geq \frac{c_1}{4}l^{\wt{\beta}_0}\right) \\
&\leq& \sum_{l\geq h_*,l\in\mathcal{H}}\bar{\mathbb{P}}\left( \left|\frac{1}{n}\sum_{i=1}^a\left(h_*^{-1}K(X_i/h_*)-\mathbb{E}_{X\sim f}h_*^{-1}K(X/h_*)\right)\right|\geq \frac{c_1}{4}\sqrt{\frac{\log n}{nh_*}}\right) \\
&& + \sum_{l\geq h_*,l\in\mathcal{H}}\bar{\mathbb{P}}\left( \left|\frac{1}{n}\sum_{i=1}^a\left(l^{-1}K(X_i/l)-\mathbb{E}_{X\sim f}l^{-1}K(X/l)\right)\right|\geq \frac{c_1}{4}\sqrt{\frac{\log n}{nl}}\right),
\end{eqnarray*}
where the last inequality is by (\ref{eq:hstarstar}) and the observation that
$$l^{\wt{\beta}_0}\geq h_*^{\wt{\beta}_0}\geq \sqrt{\frac{\log n}{nh_*}}\geq \sqrt{\frac{\log n}{nl}}.$$
Use Bernstein's inequality in a similar way that derives (\ref{eq:haoyang-nb}), we obtain the bound
$$\bar{\mathbb{P}}\left( \left|\frac{1}{n}\sum_{i=1}^a\left(l^{-1}K(X_i/l)-\mathbb{E}_{X\sim f}l^{-1}K(X/l)\right)\right|\geq \frac{c_1}{4}\sqrt{\frac{\log n}{nl}}\right)\leq 2n^{-3},$$
when the constant $c_1$ is chosen to be sufficiently large. Then, we have
$$\bar{\mathbb{P}}(\wh{h}>h_*)\lesssim n^{-3}\log n.$$

On the event $\wh h\leq h_*$, the error decomposes as
\begin{align*}
|\wh f_{\wh h}(0)-f(0)|\leq |\wh f_{\wh h}(0)-\wh f_{h_*}(0)|+|\wh f_{h_*}(0)-f(0)|.
\end{align*}
Due to definition of $\wh h$, the first term is bounded as
\begin{align*}
|\wh f_{\wh h}(0)-\wh f_{h_*}(0)|\leq c_1h_*^{\wt \beta_0}.
\end{align*}
The second term uses the oracle bandwidth $h_*$. Then, we have
\begin{align*}
\bar{\mathbb{E}}(\wh f_{\wh h}(0)-f(0))^2&\leq \bar{\mathbb{E}}((\wh f_{\wh h}(0)-f(0))^2:\wh h\leq h_*)+\bar{\mathbb{E}}((\wh f_{\wh h}(0)-f(0))^2:\wh h> h_*)\\
&\lesssim \bar{\mathbb{E}}((\wh f_{\wh h}(0)-\wh f_{h_*}(0))^2:\wh h\leq h_*)+\bar{\mathbb{E}}((\wh f_{h_*}(0)-f(0))^2:\wh h\leq h_*)+n^2\bar{\mathbb{P}}(\wh h> h_*)\\
&\lesssim h_*^{2\wt \beta_0} + \frac{1}{nh_*} + h_*^{2\beta_0} + \frac{\epsilon^2}{h_*^2} + n^{-1}\log n\\
&\lesssim \left(\frac{\log n}{n}\right)^{\frac{2\wt\beta_0}{2\wt\beta_0+1}}\vee \epsilon^{\frac{2\wt\beta_0}{\wt\beta_0+1}},
\end{align*}
where we have used (\ref{eq:hstarstar}) in the last inequality. Integrating over the random variable $a$, we have
\begin{eqnarray}
\nonumber \mathbb{E}(\wh f_{\wh h}(0)-f(0))^2 &\leq& \mathbb{E}\left((\wh f_{\wh h}(0)-f(0))^2: \frac{n-a}{n}< 2\epsilon\right) + \mathbb{E}\left((\wh f_{\wh h}(0)-f(0))^2: \frac{n-a}{n}\geq 2\epsilon\right) \\
\nonumber &\lesssim& \left(\frac{\log n}{n}\right)^{\frac{2\wt\beta_0}{2\wt\beta_0+1}}\vee \epsilon^{\frac{2\wt\beta_0}{\wt\beta_0+1}} + n^2\mathbb{P}\left(\frac{n-a}{n}\geq 2\epsilon\right) \\
\label{eq:error-ac-bound0} &\lesssim& \left(\frac{\log n}{n}\right)^{\frac{2\wt\beta_0}{2\wt\beta_0+1}}\vee \epsilon^{\frac{2\wt\beta_0}{\wt\beta_0+1}}.
\end{eqnarray}

Finally, we consider the situation when $\epsilon< \frac{8\log n}{n}$. In this case, for any contamination distribution $g$, there is another $\wt{g}$ such that
$$(1-\epsilon)f+\epsilon g= \left(1-\frac{8\log n}{n}\right)f+\frac{8\log n}{n}\wt{g}.$$
See \cite{chen2015robust} for a rigorous argument of the above equality. Then, we can equivalently analyze the risk with contamination proportion $\frac{8\log n}{n}$. This leads to the error bound
\begin{equation}
\left(\frac{\log n}{n}\right)^{\frac{2\wt\beta_0}{2\wt\beta_0+1}}\vee \left(\frac{\log n}{n}\right)^{\frac{2\wt\beta_0}{\wt\beta_0+1}}\asymp \left(\frac{\log n}{n}\right)^{\frac{2\wt\beta_0}{2\wt\beta_0+1}}\asymp \left(\frac{\log n}{n}\right)^{\frac{2\wt\beta_0}{2\wt\beta_0+1}}\vee \epsilon^{\frac{2\wt\beta_0}{\wt\beta_0+1}}.\label{eq:error-ac-bound}
\end{equation}
Hence, we let $n\rightarrow \infty$ and $\epsilon\rightarrow 0$, and the proof is complete.
\end{proof}

\begin{proof}[Proof of Theorem \ref{thm:lepski3}]
For the estimator that uses (\ref{eq:h-lep-beta}), the result is a special case of Theorem \ref{thm:arbitrarylepski} by letting $\wt{\beta}_0=\beta_0$ in view of the bounds (\ref{eq:error-ac-bound0}) and (\ref{eq:error-ac-bound}). For the estimator that uses (\ref{eq:h-lep-epsilon}), the result follows the same argument as the proof of Theorem \ref{thm:lepski} .
\end{proof}

\subsection{Proofs of Lemma \ref{thm:unidentifiable} and Theorem \ref{thm:impossible}}\label{sec:pf-imp}

\begin{proof}[Proof of Lemma \ref{thm:unidentifiable}]
We use $\phi(\cdot)$ to denote the density of $N(0,1)$. Then, define
\begin{align*}
f(x)=&c_{3}\phi(c_{3}x),\\
g(x)=&\frac{c_4}{\epsilon^{\frac{1}{\wt{\beta}_0+1}}}\phi\left(\frac{c_{4}x}{\epsilon^{\frac{1}{\wt{\beta}_0+1}}}\right),\\
\wt f(x)=&(1-\epsilon)f(x)+\epsilon g(x),\\
\wt g(x)=&\phi(x).
\end{align*}
First, there exists a constant $c_3$ depending on $c_1,c_2$ such that for any $\beta_0,\wt{\beta}_0\leq c_1$ and $L_0,\wt{L}_0\geq c_2$, we have $f\in \mathcal{P}(\beta_0,L_0)\cap \mathcal{P}(\wt{\beta}_0,\wt L_0/2)$. This is due to the fact that $\phi^{(\alpha)}(x)$ is uniformly bounded for all $\alpha\leq c$ when $c$ is some constant. By definition,
\begin{align*}
\epsilon g(x)=c_{4}\epsilon^{\frac{\wt{\beta}_0}{\wt{\beta}_0+1}}\phi\left(\frac{c_{4}x}{\epsilon^{\frac{1}{\wt{\beta}_0+1}}}\right),
\end{align*}
For the same reason as above, there exists a constant $c_4$ depending on $c_1,c_2$ such that for any $\wt{\beta}_0\leq c_1$ and $\wt{L}_0\geq c_2$, we have $\epsilon g\in \Sigma(\wt{\beta}_0,\wt L_0/2)$, which then implies $\wt f\in\mathcal{P}(\wt{\beta}_0,\wt L_0)$. Now we note that
$$(1-\epsilon)f+ \epsilon g = (1-0)\wt{f}+ 0\wt{g},$$
and
$$\left|\wt{f}(0)-f(0)\right|=\epsilon|f(0)-g(0)|\geq c_0\epsilon^{\frac{\wt{\beta}_0}{\wt{\beta}_0+1}},$$
when $\epsilon$ smaller than a constant and where $c_0$ is a constant depending on $c_3,c_4$.
Thus for any estimator $\wh{f}(0)$,
$$\left[\sup_{p(\epsilon,f,g)\in\mathcal{M}(\epsilon,\beta_0,L_0)}\mathbb{E}_{p^n}\left(\wh{f}(0)-f(0)\right)^2\right]\vee\left[\sup_{p(0,f,g)\in\mathcal{M}(0,\wt{\beta}_0,\wt{L}_0)}\mathbb{E}_{p^n}\left(\wh{f}(0)-f(0)\right)^2\right]\geq c_0\epsilon^{\frac{2\wt{\beta}_0}{\wt{\beta}_0+1}},$$
by applying Lemma \ref{lem:lowerbound}.
\end{proof}

\begin{proof}[Proof of Theorem \ref{thm:impossible}]
For any constants $c_1,c_2$, let $c_0$ be the constant as guaranteed to exist in Theorem \ref{thm:unidentifiable}, and assume there exists an estimator $\wh f(0)$ which is $(c_1,c_2,c_3,r_1(\beta_0),r_2(\beta_0))$ rate adaptive. With $L_0=c_2$, we consider two models respectively with parameters $(n,\epsilon,\beta_0,L_0)$ and $(n,0,\wt\beta_0,L_0)$, with the specific values of $n,\epsilon,\beta_0,\wt\beta_0$ to be chosen later. By the definition of rate adaptivity (\ref{eq:adaptive-def}), we have:
\begin{align*}
\sup_{p(\epsilon,f,g)\in\mathcal{M}(\epsilon,\beta_0,L_0)}\mathbb{E}_{p^n}\left(\wh{f}(0)-f(0)\right)^2&\leq c_3[\epsilon^{r_2(\beta_0)}\vee n^{-r_1(\beta_0)}],\\
\sup_{p(0,f,g)\in\mathcal{M}(0,\wt\beta_0,L_0)}\mathbb{E}_{p^n}\left(\wh{f}(0)-f(0)\right)^2&\leq c_3n^{-r_1(\wt\beta_0)}.
\end{align*}
On the other hand, Theorem \ref{thm:unidentifiable} claims that for any small enough $\epsilon$, any large enough $n$, any $\beta_0,\wt\beta_0\leq c_1$, we have
\begin{align*}
\sup_{p(\epsilon,f,g)\in\mathcal{M}(\epsilon,\beta_0,L_0)}\mathbb{E}_{p^n}\left(\wh{f}(0)-f(0)\right)^2\vee \sup_{p(0,f,g)\in\mathcal{M}(0,\wt\beta_0,L_0)}\mathbb{E}_{p^n}\left(\wh{f}(0)-f(0)\right)^2& \geq c_0\epsilon^{\frac{2\wt{\beta}_0}{\wt{\beta}_0+1}}.
\end{align*}
Together this yields
\begin{equation}\label{eq:contradiction}
c_3[\epsilon^{r_2(\beta_0)}\vee n^{-r_1(\beta_0)}\vee n^{-r_1(\wt\beta_0)}]\geq c_0\epsilon^{\frac{2\wt{\beta}_0}{\wt{\beta}_0+1}}.
\end{equation}
Now we choose $n,\beta_0,\wt\beta_0,\epsilon$ in a legitimate range so that this inequality becomes a contradiction. First we fix some $\beta_0\leq c_1$. Then we choose $\epsilon$ to be small enough such that $c_3\epsilon^{r_2(\beta_0)}\leq c_0 \epsilon^a$ for some $a>0$. Indeed this holds as long as $\epsilon^{r_2(\beta_0)}<\frac{c_0}{c_3}$. Now since $a>0$, we can choose $\wt\beta_0$ to be small enough such that $\frac{2\wt{\beta}_0}{\wt{\beta}_0+1}<a$. Finally since $r_1(\beta_0),r_1(\wt\beta_0)>0$, we can choose $n$ large enough such that $c_3[n^{-r_1(\beta_0)}\vee n^{-r_1(\wt{\beta}_0)}]< c_0\epsilon^{\frac{2\wt{\beta}_0}{\wt{\beta}_0+1}}$. With these choices, it is obvious that equation (\ref{eq:contradiction}) becomes a contradiction, as desired.
\end{proof}

\subsection{Proof of Theorem \ref{thm:multi}}

The proofs are exactly the same as in the one-dimensional case. For the lower bounds, we only need to replace the mollifier function $l(x)$ by its multivariate extension $l_d(x)=l(\|x\|)$. The upper bounds are achieved by
$\wh{f}_h(0)=\frac{1}{n(1-\epsilon)}\sum_{i=1}^n\frac{1}{h^d}K_d\left(\frac{X_i}{h}\right)$,
where the bandwidth is $h=n^{-\frac{1}{2\beta_0+d}}\wedge n^{-\frac{1}{2\beta_0+d}}\epsilon^{-\frac{2}{2\beta_0+d}}$ for structured contamination and is $h=n^{-\frac{1}{2\beta_0+d}}\vee \epsilon^{\frac{1}{\beta_0+d}}$ for arbitrary contamination. We can use a product kernel for $K_d$. See \cite[Chapter 12]{devroyecombinatorial} for details.

\section*{Acknowledgement}

The authors thank Zhao Ren for reading the manuscript and for his helpful comments.
The research of CG is supported in part by NSF grant DMS-1712957.

\bibliographystyle{plainnat}
\bibliography{Robust}

\begin{thebibliography}{24}
\providecommand{\natexlab}[1]{#1}
\providecommand{\url}[1]{\texttt{#1}}
\expandafter\ifx\csname urlstyle\endcsname\relax
  \providecommand{\doi}[1]{doi: #1}\else
  \providecommand{\doi}{doi: \begingroup \urlstyle{rm}\Url}\fi

\bibitem[Brown and Low(1996)]{brown1996constrained}
Lawrence~D Brown and Mark~G Low.
\newblock A constrained risk inequality with applications to nonparametric
  functional estimation.
\newblock \emph{The Annals of Statistics}, 24\penalty0 (6):\penalty0
  2524--2535, 1996.

\bibitem[Cai(2003)]{cai2003rates}
T~Tony Cai.
\newblock Rates of convergence and adaptation over besov spaces under pointwise
  risk.
\newblock \emph{Statistica Sinica}, 13:\penalty0 881--902, 2003.

\bibitem[Cai and Jin(2010)]{cai2010optimal}
T~Tony Cai and Jiashun Jin.
\newblock Optimal rates of convergence for estimating the null density and
  proportion of nonnull effects in large-scale multiple testing.
\newblock \emph{The Annals of Statistics}, 38\penalty0 (1):\penalty0 100--145,
  2010.

\bibitem[Cai and Low(2005)]{cai2005adaptive}
T~Tony Cai and Mark~G Low.
\newblock On adaptive estimation of linear functionals.
\newblock \emph{The Annals of Statistics}, 33\penalty0 (5):\penalty0
  2311--2343, 2005.

\bibitem[Cai and Low(2006)]{cai2006optimal}
T~Tony Cai and Mark~G Low.
\newblock Optimal adaptive estimation of a quadratic functional.
\newblock \emph{The Annals of Statistics}, 34\penalty0 (5):\penalty0
  2298--2325, 2006.

\bibitem[Chen et~al.(2016)Chen, Gao, and Ren]{chen2016general}
Mengjie Chen, Chao Gao, and Zhao Ren.
\newblock A general decision theory for huber’s $\epsilon$-contamination
  model.
\newblock \emph{Electronic Journal of Statistics}, 10\penalty0 (2):\penalty0
  3752--3774, 2016.

\bibitem[Chen et~al.(2017)Chen, Gao, and Ren]{chen2015robust}
Mengjie Chen, Chao Gao, and Zhao Ren.
\newblock Robust covariance matrix estimation under huber’s contamination
  model.
\newblock \emph{The Annals of Statistics (to appear)}, 2017.

\bibitem[Devroye and Lugosi(2001)]{devroyecombinatorial}
L~Devroye and G~Lugosi.
\newblock Combinatorial methods in density estimation, 2001.

\bibitem[Donoho(1994)]{donoho1994statistical}
David~L Donoho.
\newblock Statistical estimation and optimal recovery.
\newblock \emph{The Annals of Statistics}, 22\penalty0 (1):\penalty0 238--270,
  1994.

\bibitem[Donoho and Liu(1991)]{donoho1991geometrizing}
David~L Donoho and Richard~C Liu.
\newblock Geometrizing rates of convergence, iii.
\newblock \emph{The Annals of Statistics}, 19\penalty0 (2):\penalty0 668--701,
  1991.

\bibitem[Efron(2004)]{efron2004large}
Bradley Efron.
\newblock Large-scale simultaneous hypothesis testing: the choice of a null
  hypothesis.
\newblock \emph{Journal of the American Statistical Association}, 99\penalty0
  (465):\penalty0 96--104, 2004.

\bibitem[Gao(2017)]{gao2017robust}
Chao Gao.
\newblock Robust regression via mutivariate regression depth.
\newblock \emph{arXiv preprint arXiv:1702.04656}, 2017.

\bibitem[Huber(1964)]{huber1964robust}
Peter~J Huber.
\newblock Robust estimation of a location parameter.
\newblock \emph{The Annals of Mathematical Statistics}, 35\penalty0
  (1):\penalty0 73--101, 1964.

\bibitem[Huber(1965)]{huber1965robust}
Peter~J Huber.
\newblock A robust version of the probability ratio test.
\newblock \emph{The Annals of Mathematical Statistics}, 36\penalty0
  (6):\penalty0 1753--1758, 1965.

\bibitem[Jin and Cai(2007)]{jin2007estimating}
Jiashun Jin and T~Tony Cai.
\newblock Estimating the null and the proportion of nonnull effects in
  large-scale multiple comparisons.
\newblock \emph{Journal of the American Statistical Association}, 102\penalty0
  (478):\penalty0 495--506, 2007.

\bibitem[Johnstone(2001)]{johnstone2001chi}
Iain~M Johnstone.
\newblock Chi-square oracle inequalities.
\newblock \emph{Lecture Notes-Monograph Series}, pages 399--418, 2001.

\bibitem[Lepski and Spokoiny(1997)]{lepski1997optimal}
OV~Lepski and VG~Spokoiny.
\newblock Optimal pointwise adaptive methods in nonparametric estimation.
\newblock \emph{The Annals of Statistics}, 25\penalty0 (6):\penalty0
  2512--2546, 1997.

\bibitem[Lepskii(1991)]{lepskii1991problem}
OV~Lepskii.
\newblock On a problem of adaptive estimation in gaussian white noise.
\newblock \emph{Theory of Probability \& Its Applications}, 35\penalty0
  (3):\penalty0 454--466, 1991.

\bibitem[Lepskii(1992)]{lepskii1992asymptotically}
OV~Lepskii.
\newblock Asymptotically minimax adaptive estimation. i: Upper bounds.
  optimally adaptive estimates.
\newblock \emph{Theory of Probability \& Its Applications}, 36\penalty0
  (4):\penalty0 682--697, 1992.

\bibitem[Lepskii(1993)]{lepskii1993asymptotically}
OV~Lepskii.
\newblock Asymptotically minimax adaptive estimation. ii. schemes without
  optimal adaptation: Adaptive estimators.
\newblock \emph{Theory of Probability \& Its Applications}, 37\penalty0
  (3):\penalty0 433--448, 1993.

\bibitem[Silverman(1986)]{silverman1986density}
Bernard~W Silverman.
\newblock \emph{Density estimation for statistics and data analysis},
  volume~26.
\newblock CRC press, 1986.

\bibitem[Tribouley(2000)]{tribouley2000adaptive}
Karine Tribouley.
\newblock Adaptive estimation of integrated functionals.
\newblock \emph{Mathematical Methods of Statistics}, 9\penalty0 (1):\penalty0
  19--38, 2000.

\bibitem[Tsybakov(2009)]{tsybakov09}
Alexandre~B Tsybakov.
\newblock \emph{Introduction to nonparametric estimation}, volume~11.
\newblock Springer, 2009.

\bibitem[Yu(1997)]{yu1997assouad}
Bin Yu.
\newblock Assouad, fano, and le cam.
\newblock \emph{Festschrift for Lucien Le Cam}, 423:\penalty0 435, 1997.

\end{thebibliography}


\end{document}